\documentclass[11pt,a4paper]{article}
\pdfoutput=1 

\usepackage{epsf,epsfig,amsfonts,amsgen,amsmath,amstext,amsbsy,amsopn,amsthm,cases,listings,color
}
\usepackage{ebezier,eepic}
\usepackage{color}
\usepackage{multirow}
\usepackage{epstopdf}
\usepackage{graphicx}
\usepackage{pgf,tikz}
\usepackage{mathrsfs}
\usepackage[marginal]{footmisc}
\usepackage{enumitem}
\usepackage[titletoc]{appendix}
\usepackage{booktabs}
\usepackage{url}
\usepackage{mathtools}
\usepackage{comment}
\usepackage{pgfplots}
\usepackage{authblk}
\usepackage{amssymb}
\usepackage{float}
\usepackage{wasysym}

\usepackage{empheq}

\usepackage{dsfont}

\usepackage{tikz}
\usepackage{longtable}
\usepackage{subfigure}
\pgfplotsset{compat=1.18}
\usepackage{mathrsfs}
\usepackage{wasysym}
\usetikzlibrary{arrows}
\usepackage{aligned-overset}
\usepackage{bm}
\usepackage{bbm}
\usepackage[T1]{fontenc}
\usepackage[backref=page]{hyperref}

\allowdisplaybreaks[1]

\definecolor{uuuuuu}{rgb}{0.27,0.27,0.27}
\definecolor{sqsqsq}{rgb}{0.1255,0.1255,0.1255}

\newtheorem{definition}{Definition} [section]
\newtheorem{theorem}[definition]{Theorem}
\newtheorem{lemma}[definition]{Lemma}
\newtheorem{proposition}[definition]{Proposition}

\newtheorem{corollary}[definition]{Corollary}
\newtheorem{conjecture}[definition]{Conjecture}
\newtheorem{claim}[definition]{Claim}
\newtheorem{problem}[definition]{Problem}

\newtheorem{fact}[definition]{Fact}

\newenvironment{pf}{\noindent{\bf Proof}}

\newcommand{\norm}[1]{\left\lVert#1\right\rVert}

\newcommand{\h}{\mathcal{H}}

\newcommand{\G}{\mathcal{G}}

\tikzset{unlabeled_vertex/.style={inner sep=1.7pt, outer sep=0pt, circle, fill}}
\tikzset{labeled_vertex/.style={inner sep=2.2pt, outer sep=0pt, rectangle, fill=yellow, draw=black}}
\tikzset{edge_color0/.style={color=black,line width=1.2pt,opacity=0.5}}
\tikzset{edge_color1/.style={color=red,  line width=1.2pt,opacity=1}}
\tikzset{edge_color2/.style={color=blue, line width=1.2pt,opacity=1}}
\tikzset{edge_color3/.style={color=green,line width=1.2pt}}
\tikzset{edge_color4/.style={color=red,  line width=1.2pt,dotted}}
\tikzset{edge_color5/.style={color=blue, line width=1.2pt,dotted}}
\tikzset{edge_color6/.style={color=green, line width=1.2pt,dotted}}
\tikzset{edge_color7/.style={color=orange, line width=1.2pt}}
\tikzset{edge_color8/.style={color=gray, line width=1.2pt}}
\tikzset{edge_thin/.style={color=black}}
\tikzset{edge_hidden/.style={color=black,dotted,opacity=0}}
\tikzset{vertex_color1/.style={inner sep=1.7pt, outer sep=0pt, draw, circle, fill=red}}
\tikzset{vertex_color2/.style={inner sep=1.7pt, outer sep=0pt, draw, circle, fill=blue}}
\tikzset{vertex_color3/.style={inner sep=1.7pt, outer sep=0pt, draw, circle, fill=green}}
\tikzset{labeled_vertex_color1/.style={inner sep=2.2pt, outer sep=0pt, draw, rectangle, fill=red}}
\tikzset{labeled_vertex_color2/.style={inner sep=2.2pt, outer sep=0pt, draw, rectangle, fill=blue}}
\tikzset{labeled_vertex_color3/.style={inner sep=2.2pt, outer sep=0pt, draw, rectangle, fill=green}}

\tikzset{
vtx/.style={inner sep=1.1pt, outer sep=0pt, circle, fill,draw},
vtxl/.style={inner sep=1.1pt, outer sep=0pt, rectangle, fill=yellow,draw=black},
hyperedge/.style={fill=pink,opacity=0.5,draw=black},
}

\begin{document}
\title{\bf\Large Phase transition of degenerate Tur\'{a}n problems in $p$-norms}
\date{\today}
\author[1]{Jun Gao\thanks{Research  supported by IBS-R029-C4. Email: \texttt{jungao@ibs.re.kr}}}
\author[2]{Xizhi Liu\thanks{Research  supported by ERC Advanced Grant 101020255. Email: \texttt{xizhi.liu.ac@gmail.com}}}
\author[3]{Jie Ma\thanks{Research supported by National Key Research and Development Program of China 2023YFA1010201 and National Natural Science Foundation of China grant 12125106. Email: \texttt{jiema@ustc.edu.cn}}}
\author[2]{Oleg Pikhurko\thanks{Research  supported by ERC Advanced Grant 101020255. Email: \texttt{o.pikhurko@warwick.ac.uk}}}
\affil[1]{Extremal Combinatorics and Probability Group (ECOPRO),
Institute for Basic Science (IBS),
Daejeon,
South Korea.}
\affil[2]{Mathematics Institute and DIMAP,
            University of Warwick,
            Coventry, CV4 7AL, UK}
\affil[3]{School of Mathematical Sciences,
            University of Science and Technology of China, 
            Hefei, Anhui, 230026, China}
\maketitle
\begin{abstract}
For a positive real number $p$, the $p$-norm $\norm{G}_p$ of a graph $G$ is the sum of the $p$-th powers of all vertex degrees. 
We study the maximum $p$-norm $\mathrm{ex}_{p}(n,F)$ of $F$-free graphs on $n$ vertices.
F\"{u}redi and K\"{u}ndgen \cite{FK06}  show that for every bipartite graph $F$, there exists a threshold $p_F$ such that for $p< p_{F}$, the order of $\mathrm{ex}_{p}(n,F)$ is governed by pseudorandom constructions, while for $p > p_{F}$, it is governed by star-like constructions, assuming a mild assumption on the growth rate of $\mathrm{ex}(n,F)$.
The main contribution of our paper is extending this result to hypergraph. Moreover, in the case of graph, our proof differs  from that in \cite{FK06}, offering the advantage of producing the correct constant factor when $p > p_{F}$.

When \( p = p_F \), F\"{u}redi and K\"{u}ndgen proved a general upper bound on \( \mathrm{ex}_{p}(n,F) \), tight up to a \( \log n \) factor, and conjectured that this factor is unnecessary. We confirm this conjecture for several well-studied bipartite graphs, including one-side degree-bounded graphs and families of short even cycles.  

\medskip



\end{abstract}
\section{Introduction}
Given an integer $r\ge 2$, an \textbf{$r$-uniform hypergraph} (henceforth an \textbf{$r$-graph}) on a set $V$ is a subset $\mathcal{H}$ of $\binom{V}{r}\coloneqq \{ X\subseteq V: |X|=r \}$.
We identify a hypergraph $\mathcal{H}$ with its edge set and use $V(\mathcal{H})$ to denote its vertex set. 
The size of $V(\mathcal{H})$ is denoted by $v(\mathcal{H})$. 
The \textbf{degree} $d_{\mathcal{H}}(v)$ of $v$ in $\mathcal{H}$ is the number of edges in $\mathcal{H}$ containing $v$.


Given an $r$-graph $\mathcal{H}$ and a real number $p\ge 0$, let the \textbf{$p$-norm} of $\mathcal{H}$ be defined as
\begin{align*}
    \norm{\mathcal{H}}_{p}
    \coloneqq \sum_{v\in V(\mathcal{H})} d_{\mathcal{H}}^{p}(v),
\end{align*}
where, for convenience, we write $d_{\mathcal{H}}^{p}(v) \coloneqq \left(d_{\mathcal{H}}(v)\right)^{p}$.

Given a family $\mathcal{F}$ of $r$-graphs, we say an $r$-graph $\mathcal{H}$ is \textbf{$\mathcal{F}$-free}
if it does not contain any member of $\mathcal{F}$ as a subgraph.
The \textbf{$p$-norm Tur\'{a}n number} of $\mathcal{F}$ is defined as 
\begin{align*}
    \mathrm{ex}_{p}(n,\mathcal{F})
    \coloneqq \max\left\{\norm{\mathcal{H}}_{p} \colon \text{$v(\mathcal{H}) = n$ and $\mathcal{H}$ is $\mathcal{F}$-free}\right\}.
\end{align*}
The case $p=1$ corresponds to the \textbf{Tur\'{a}n number} $\mathrm{ex}(n, \mathcal{F})$ of $\mathcal{F}$ (differing only by a multiplicative factor of $r$), which represents the maximum number of edges in an $n$-vertex $\mathcal{F}$-free $r$-graph.

Extending the seminal work of Tur\'{a}n~\cite{Tur41}, Caro--Yuster~\cite{CY00,CY00arxiv} initiated\footnote{According to the Introduction in~\cite{FK06}, it seems that $\mathrm{ex}_{p}(n,K_t)$ was already considered by Erd{\H o}s in the 1970s (see~\cite{Erdos70}).} the study of $p$-norm Tur\'{a}n problem for graphs by determining the value of $\mathrm{ex}_{p}(n,K_{\ell+1})$ for $p \ge 1$.
This line of research has since been extended to various other graphs and hypergraphs, as explored in works such as~\cite{Nik09,BN12,LLQS19,BCL22,BCL22b,Zha22,Ger24,CIDLLP24}.
In this work, we focus on the case where $\mathcal{F}$ is degenerate. 

The \textbf{Tur\'{a}n density} of $\mathcal{F}$ is defined as $\pi(\mathcal{F})\coloneq \lim_{n\to\infty}\mathrm{ex}(n,\mathcal{F})/{n\choose r}$. 
A family $\mathcal{F}$ of $r$-graphs is called \textbf{degenerate} if $\pi(\mathcal{F}) = 0$. 
According to a classical theorem of Erd\H{o}s~\cite{Erdos64}, this is equivalent to stating that $\mathcal{F}$ contains at least one $r$-partite $r$-graph. 
Determining the growth rate of $\mathrm{ex}(n, \mathcal{F})$ for degenerate families is a central and notoriously difficult topic in Extremal Combinatorics, and it remains unresolved for most families.
For example, the Even Cycle Problem proposed by Erd\H{o}s~\cite{E64,BS74}, which asks for the exponent of $\mathrm{ex}(n,C_{2k})$, is still open for every $k$ not in $\{2,3,5\}$ (see e.g.~\cite{ERS66,Ben66,Wen91,LU93,LUW99}).  
For more results on degenerate Tur\'{a}n problems,  we refer the reader to the survey~\cite{FS13}.

For an $r$-partite $r$-graph $F$, the \textbf{partition number} $\tau_{\mathrm{part}}(F)$ of $F$ is defined as the the minimum size of a set $S_1\subseteq V(F)$ such that $V(F)\setminus S_1$ can be partitioned into $r-1$ sets $S_2, \ldots, S_{r}$, with each edge of $F$ containing exactly one vertex from each $S_i$.
The \textbf{independent covering number} $\tau_{\mathrm{ind}}(F)$ of $F$ is defined as the minimum size of a set $S$ such that every edge of $F$ contains exactly one vertex from $S$.
It is clear from the definition that  $\tau_{\mathrm{ind}}(F) \le \tau_{\mathrm{part}}(F)$ for every $r$-partite $r$-graph $F$, and $\tau_{\mathrm{ind}}(F) = \tau_{\mathrm{part}}(F)$ for every bipartite graph $F$.

Given the definitions that we have introduced, we can immediately derive the following two general lower bounds for $\mathrm{ex}_{p}(n,F)$. 
\begin{fact}\label{FACT:r-gp-p-norm-lower-bound}
    Let $r \ge 2$ be an integer and $F$ be an $r$-partite $r$-graph. 
    For every real number $p\geq 1$, we have 
    \begin{align*}
        \mathrm{ex}_{p}(n,F)
        \ge \max\left\{n \left(\frac{r\cdot \mathrm{ex}(n,F)}{n}\right)^{p},~ \left(\tau_{\mathrm{ind}}(F) - 1\right) \binom{n - \tau_{\mathrm{ind}}(F) + 1}{r-1}^{p}\right\}. 
    \end{align*}
\end{fact}
The first lower bound arises from the construction for $\mathrm{ex}(n,\mathcal{F})$ as well as convexity (see Corollary~\ref{CORO:p-q-norm-Holder}). 
The second lower bound is based on the star-like $r$-graph $S^{r}(n,t)$ for $t=\tau_{\mathrm{ind}}(F)-1$, where
\begin{align*}
    S^{r}(n,t)
    \coloneqq \left\{e\in \binom{[n]}{r} \colon |e\cap [t]| = 1\right\}, \quad \text{and} \quad [n]\coloneqq \{1,\dots,n\}.
\end{align*}

Our work is motivated by the combination of the following facts in graphs.
For $p=1$, the lower bound constructions for $\mathrm{ex}_{1}(n,\mathcal{F})$ often exhibit certain pseudorandom properties (see e.g.~\cite{KRS96,ARS99,MYZ18,PZ21}) and, in particular, are almost regular, meaning that the maximum and minimum degrees differ by only a constant factor. 
In contrast, works of Caro--Yuster~\cite{CY00}, Nikiforov~\cite{Nik09}, and Gerbner~\cite{Ger24} on even cycles and complete bipartite graphs show that for large $p$, the lower bound construction for $\mathrm{ex}_{p}(n,\mathcal{F})$ are highly structured and resemble $S^{2}(n,t)$ for some appropriate choice of $t$.  

This contrast suggests that a general phenomenon (see Figure~\ref{fig:Phi-k-6}) may hold$\colon$
for every degenerate family $\mathcal{F}$ of $r$-graphs with $\mathrm{ex}(n,\mathcal{F}) = \Omega(n^{1+\alpha})$ for some $\alpha > r-2$, there exists a threshold $p_{\mathcal{F}} > 1$ such that, for $p \in (1, p_{\mathcal{F}})$, $\mathrm{ex}_{p}(n,\mathcal{F}) = O\left(n \left(\frac{\mathrm{ex}(n,\mathcal{F})}{n}\right)^{p}\right)$, while for $p > p_{\mathcal{F}}$, $\mathrm{ex}_{p}(n,\mathcal{F}) = O\left(n^{p(r-1)}\right)$. 
F\"{u}redi and K\"{u}ndgen \cite{FK06} show that this holds for $r=2$. In the following theorem, we show that this holds for all $r\ge 2$.

\begin{figure}[H]
\centering
\begin{tikzpicture}[xscale=7,yscale=7]
\draw [line width=1pt, ->] (0,0)--(0.73,0);
\draw [line width=1pt, ->] (0,0)--(0,0.7);
%
\draw[line width=1.2pt, color=magenta, dash pattern=on 1pt off 1.2pt] (0, 1/15) -- (1/10, 1/10) -- (1/5, 1/5) -- (3/5+1/30, 3/5+1/30);
\draw[line width=1.2pt, color=cyan, dash pattern=on 0.8pt off 1pt] (0, 1/10) --  (1/5, 1/5) -- (3/5+1/30, 3/5+1/30);
\draw[line width=1.2pt, color=green, dash pattern=on 0.6pt off 0.8pt] (0, 2/15) -- (1/5, 4/15) -- (2/5, 2/5) -- (3/5+1/30, 3/5+1/30);
\draw[line width=1.2pt, color=green, dash pattern=on 0.8pt off 1pt] (0.9, 0.64) -- (1.02, 0.64); 
\draw[line width=1.2pt, color=cyan, dash pattern=on 0.8pt off 1pt] (0.9, 0.55) -- (1.02, 0.55); 
\draw[line width=1.2pt, color=magenta, dash pattern=on 0.8pt off 1pt] (0.9, 0.46) -- (1.02, 0.46); 
\begin{small}
\draw[color=uuuuuu] (1.02+0.17,0.59+0.06) node {$\frac{\log \mathrm{ex}_{p}(n,K_{3,3})}{\log n}$};
\draw[color=uuuuuu] (1.02+0.17,0.5+0.06) node {$\frac{\log \mathrm{ex}_{p}(n,C_4)}{\log n}$};
\draw[color=uuuuuu] (1.02+0.17,0.41+0.06) node {$\frac{\log \mathrm{ex}_{p}(n,C_6)}{\log n}$};
\draw [fill=uuuuuu] (0, 1/10) circle (0.2pt);
\draw [fill=uuuuuu] (0, 1/15) circle (0.2pt);
\draw [fill=uuuuuu] (1/10, 1/10) circle (0.2pt);
\draw [fill=uuuuuu] (0, 2/15) circle (0.2pt);
\draw [fill=uuuuuu] (1/5, 1/5) circle (0.2pt);
\draw [fill=uuuuuu] (2/5, 2/5) circle (0.2pt);
\draw [fill=uuuuuu] (0, 1/5) circle (0.2pt);
\draw[color=uuuuuu] (0-0.05,1/5) node {$2$};
\draw [fill=uuuuuu] (0, 2/5) circle (0.2pt);
\draw[color=uuuuuu] (0-0.05,2/5) node {$3$};
\draw [fill=uuuuuu] (0, 3/5) circle (0.2pt);
\draw[color=uuuuuu] (0-0.05,3/5) node {$4$};
\draw [fill=uuuuuu] (0,0) circle (0.2pt);
\draw[color=uuuuuu] (0,0-0.05) node {$1$};
\draw [fill=uuuuuu] (1/5,0) circle (0.2pt);
\draw[color=uuuuuu] (1/5,0-0.05) node {$2$};
\draw [fill=uuuuuu] (2/5,0) circle (0.2pt);
\draw[color=uuuuuu] (2/5,0-0.05) node {$3$};
\draw [fill=uuuuuu] (3/5,0) circle (0.2pt);
\draw[color=uuuuuu] (3/5,0-0.05) node {$4$};
\draw[color=uuuuuu] (0.65+0.1,0-0.05) node {$p$};
\end{small}
\end{tikzpicture}
\caption{Exponents of $\mathrm{ex}_{p}(n,K_{3,3})$, $\mathrm{ex}_{p}(n,C_{4})$, and $\mathrm{ex}_{p}(n,C_{6})$.}
\label{fig:Phi-k-6}
\end{figure}

For a family $\mathcal{F}$ of $r$-graphs, we define \begin{align*}
    \tau_{\mathrm{part}}(\mathcal{F})
    & \coloneqq \min\left\{\tau_{\mathrm{part}}(F) \colon \text{$F \in \mathcal{F}$ is $r$-partite}\right\}
    \quad\text{and}\quad \\
    \tau_{\mathrm{ind}}(\mathcal{F})
    & \coloneqq \min\left\{\tau_{\mathrm{ind}}(F) \colon \text{$F \in \mathcal{F}$ is $r$-partite}\right\}.
\end{align*}
\begin{theorem}\label{THM:main-r-graph}
    Let $r \ge 2$ be an integer and $p > 1$ be a real number. 
    Suppose that $\mathcal{F}$ is a degenerate family of $r$-graphs satisfying $\mathrm{ex}(n,\mathcal{F}) = O(n^{1+\alpha})$ for some constant $\alpha \in [r-2, r-1)$. 
    Then there exists a constant $C_{\mathcal{F}}>0$ such that 
    \begin{align*}
        \mathrm{ex}_{p}(n,\mathcal{F})
        \le  \begin{cases}
            C_{\mathcal{F}} \cdot n^{1+p \alpha}, & \text{if}\quad 1 < p < \frac{1}{r-1-\alpha}, \\
            \left(\tau_{\mathrm{part}}(\mathcal{F}) -1+o(1)\right)\binom{n}{r-1}^{p}, & \text{if}\quad p > \frac{1}{r-1-\alpha}.
        \end{cases}
    \end{align*}
    %
    In particular, for $r=2$, we have, for every $p > \frac{1}{1-\alpha}$, 
    \begin{align*}
        \mathrm{ex}_{p}(n,\mathcal{F})
        = \left(\tau_{\mathrm{ind}}(\mathcal{F})-1+o(1)\right)n^{p}. 
    \end{align*}
\end{theorem}
\textbf{Remarks.}
\begin{itemize}
    \item The Rational Exponent Conjecture of Erd{\H o}s--Simonovits (see~{\cite[Conjecture~1.6]{FS13}}) states that for every degenerate finite family $\mathcal{F}$ of graphs, there exist a rational number $\alpha$ and a constant $c>0$ such that
    \begin{align*}
        \lim_{n\to \infty}\frac{\mathrm{ex}(n,\mathcal{F})}{n^{1+\alpha}} = c.
    \end{align*}
    Note that by Corollary~\ref{CORO:p-q-norm-Holder}, if this conjecture holds,  then Theorem~\ref{THM:main-r-graph} is tight in the exponent for every $p \in \left(1, \frac{1}{r-1-\alpha}\right)$ when $r=2$.
    \item If $\mathrm{ex}(n,\mathcal{F}) = O(n^{1+\beta})$ for some $\beta \le r-2$, then by taking $\alpha = r-2$ in Theorem~\ref{THM:main-r-graph}, we obtain $\frac{1}{r-1-\alpha} = 1$, and hence, 
    \begin{align*}
        \mathrm{ex}_{p}(n,\mathcal{F}) 
        \le \left(\tau_{\mathrm{part}}(\mathcal{F}) -1+o(1)\right)\binom{n}{r-1}^{p}
        \quad\text{for every}\quad p \ge 1. 
    \end{align*}
    This bound is tight in the exponent unless $\mathcal{F}$ contains an $r$-graph $F$ with $\tau_{\mathrm{part}}(F) = 1$. In that case, it is straightforward to show that, for $r=2$, either $\mathrm{ex}_{p}(n,\mathcal{F}) = \Theta(n)$ (if $\mathrm{ex}(n,\mathcal{F})  = \Theta(n)$) or $\mathrm{ex}_{p}(n,\mathcal{F}) = \Theta(1)$ (if $\mathrm{ex}(n,\mathcal{F}) = \Theta(1)$) for every $p \ge 1$. The case $r \ge 3$ seems to be more complex, even in the special case of intersection problems (when each forbidden $r$-graph has only $2$ edges), see~\cite{FT16} for a survey.
\end{itemize}

For $p$ at the threshold, i.e. for $p = \frac{1}{r-1-\alpha}$, 
F\"{u}redi and K\"{u}ndgen \cite{FK06} prove a general upper bound that is tight up to a $\log n$ factor for $\mathrm{ex}_{p}(n,F)$ when $r=2$. In the following theorem, we generalize this result to $r\ge 3$, where the base of $\log$ is assumed to be $2$.
%
\begin{theorem}\label{THM:r-graph-critical-point}
    Let $r \ge 2$ be an integer. 
    Suppose that $\mathcal{F}$ is a degenerate family of $r$-graphs satisfying $\mathrm{ex}(n,\mathcal{F}) = O(n^{1+\alpha})$ for some constant $\alpha \ge r-2$. 
    Then 
    \begin{align*}
        \mathrm{ex}_{p_{\ast}}(n,\mathcal{F})
        = O\left(n^{p_{\ast}(r-1)} \log n\right)
        \quad\text{where}\quad p_{\ast}\coloneqq \frac{1}{r-1-\alpha}.
    \end{align*}
\end{theorem}
We conjecture that the $\log n$ factor in Theorem~\ref{THM:r-graph-critical-point} can be removed, thus extending the conjecture of F\"{u}redi and K\"{u}ndgen \cite{FK06} who made it for $r=2$.
In support of this conjecture, we prove it for several well-studied families of bipartite graphs in the following theorem. 

Given a bipartite graph $F$ with two parts $V_1$ and $V_2$, we say $F$ is \textbf{$s$-bounded} if every vertex in $V_2$ has degree at most $s$.
A celebrated theorem of F{\" u}redi~\cite{Furedi91}, later refined by  Alon--Krivelevich--Sudakov~\cite{AKS03}, establishes that $\mathrm{ex}(n,F) = O(n^{2-\frac{1}{s}})$ for every $s$-bounded bipartite graph $F$. 
This bound is tight for graphs such as complete bipartite graphs $K_{s,t}$ when $t$ is sufficiently large~\cite{KRS96,ARS99,Bukh21}.

\begin{theorem}\label{THM:graph-critical-point}
    The following statements hold for sufficiently large $n$. 
    \begin{enumerate}[label=(\roman*)]
        \item\label{THM:graph-critical-point-1} $\mathrm{ex}_{\ell/(\ell-1)}(n,\{C_{4}, C_{6}, \ldots, C_{2\ell}\}) \le  765 n^{\frac{\ell}{\ell-1}}$ for every $\ell \ge 3$.
        \item\label{THM:graph-critical-point-2} $\mathrm{ex}_{3/2}(n,C_{6}) \le 2164 n^{3/2}$.
        \item\label{THM:graph-critical-point-3} Suppose that $F$ is an $s$-bounded bipartite graph. Then 
        \begin{align*}
            \mathrm{ex}_{s}(n,F) 
            \le 2 \left(\frac{|V(F)|^{s}}{s!} + |V(F)|\right)  n^{s}. 
        \end{align*}
    \end{enumerate}
\end{theorem}
In the following section, we present some preliminary results. 
In Section~\ref{SEC:regular-lemma}, we introduce a $p$-norm extension of the classical $\Delta$-almost-Regularization Theorem by Erd{\H o}s--Simonovits. 
The proofs of Theorems~\ref{THM:main-r-graph},~\ref{THM:r-graph-critical-point} and~\ref{THM:graph-critical-point} are provided in  Sections~\ref{SEC:proof-r-gp},~\ref{SEC:proof-r-gh-critical}, and~\ref{SEC:proof-graph-critical}, respectively. 
Section~\ref{SEC:remarks} includes some open problems and concluding remarks.

\textbf{Remark.}
After the preprint was posted on arXiv, D\'{a}niel Gerbner informed us that results similar to Theorems~\ref{THM:main-r-graph} and~\ref{THM:graph-critical-point} for the case $r=2$ were already proved by F\"{u}redi--K\"{u}ndgen in~{\cite[Theorem~3.3]{FK06}} using an elegant and concise argument.
Our proofs of both theorems appear to be quite different from the approach taken by F\"{u}redi--K\"{u}ndgen. In the case $p < 1/(1-\alpha)$, our proof relies on a $p$-norm adaption of the classical $\Delta$-almost-Regularization Theorem by Erd{\H o}s--Simonovits, which is of independent interest.  In the case $p > 1/(1-\alpha)$, our proof has the additional advantage of providing the tight main term.
\section{Preliminaries}\label{SEC:Prelim}
We present some notation and preliminary results 
that will be used in the subsequent proofs.

Given an $r$-graph $\mathcal{H}$,
we use $\delta(\mathcal{H})$, $\Delta(\mathcal{H})$, and $d(\mathcal{H})$ to denote the \textbf{minimum}, \textbf{maximum}, and \textbf{average degree} of $\mathcal{H}$, respectively.
For a vertex $v\in V(\mathcal{H})$, the \textbf{link} $L_\mathcal{H}(v)$ of $v$ is defined as the $(r-1)$-graph consisting of all $(r-1)$-sets $S$ such that $S\cup \{v\}\in \mathcal{H}$.
We will omit the subscript $\mathcal{H}$ if it is clear from the context.

Unless otherwise stated, all asymptotic notations in this paper are considered with respect to $n$. Floors and ceilings will be omitted when they are not critical to the proofs.
The base of $\log$ is assumed to be $2$.
\begin{fact}\label{FACT:inequality-p-sum}
    Let $p \ge 1$ and $x \ge y \ge 0$ be real numbers.
    Then 
    \begin{align*}
        \left(x^p + y^p\right)^{1/p}
        \ge x 
        \ge \frac{x+y}{2}.
    \end{align*}
\end{fact}

\begin{fact}[Power Mean Inequality]\label{FACT:Holder}
    Let $p > q \geq 1$ be two real numbers and $x_1,\ldots, x_n$ be non-negative real numbers. 
    Then 
    \begin{align*}
        \left(\frac{\sum_{i\in [n]}x_i^p}{n}\right)^{1/p} 
        \ge \left(\frac{\sum_{i\in [n]}x_i^q}{n}\right)^{1/q}.
    \end{align*}
\end{fact}

\begin{fact}[Minkowski's inequality]\label{FACT:Minkowski}
    Let $p \ge 1$ and $x_1,\ldots, x_n, y_1, \ldots, y_n$ be real numbers. 
    Then 
    \begin{align*}
        \left(\sum_{i=1}^n|x_i+y_i|^p\right)^{1/p} 
        \le \left(\sum_{i=1}^n|x_i|^p\right)^{1/p} + \left(\sum_{i=1}^n|y_i|^p\right)^{1/p}.
    \end{align*}
    In particular, for every $p \ge 1$ and $x, y \ge 0$, 
    \begin{align*}
        (x^{p} + y^{p})^{1/p} 
        \le x + y.
    \end{align*}
\end{fact}


\begin{fact}\label{FACT:p-q-norm-upper}
    Let $p > q \geq 1$ be two real numbers and $\mathcal{H}$ be an $r$-graph on $n$ vertices. 
    Then 
    \begin{align*}
        \norm{\mathcal{H}}_p
        = 
        \sum_{v\in V(\mathcal{H})} d_{\mathcal{H}}^{p}(v)
        & =\sum_{v\in V(\mathcal{H})} d_{\mathcal{H}}^{q+(p-q)}(v) \\
        & \le \sum_{v\in V(\mathcal{H})} d_{\mathcal{H}}^{q}(v) \cdot \left(\Delta(\mathcal{H})\right)^{p-q} 
        = \norm{\mathcal{H}}_q \cdot \left(\Delta(\mathcal{H})\right)^{p-q}. 
    \end{align*}
    In particular, $\norm{\mathcal{H}}_p \le \norm{\mathcal{H}}_q \cdot n^{(r-1)(p-q)}$. 
\end{fact}

The following result is an immediate corollary of Fact~\ref{FACT:Holder}. 
\begin{corollary}\label{CORO:p-q-norm-Holder}
    Let $p > q \geq 1$ be two real numbers and $\mathcal{H}$ be an $n$-vertex $r$-graph.
    Then 
    \begin{align*}
        \left(\frac{\norm{\mathcal{H}}_p}{n}\right)^{1/p} 
        \ge \left(\frac{\norm{\mathcal{H}}_q}{n}\right)^{1/q}.
    \end{align*}
    Consequently, $\norm{\mathcal{H}}_p \ge n \left(\norm{\mathcal{H}}_q/n\right)^{p/q}$, and hence, 
    \begin{align}\label{equ:Turan-number-p-lower-bound}
        \mathrm{ex}_{p}(n,\mathcal{F}) 
        \ge n \left(\frac{\mathrm{ex}_{1}(n,\mathcal{F})}{n}\right)^{p}
        = n \left(\frac{r\cdot \mathrm{ex}(n,\mathcal{F})}{n}\right)^{p}. 
    \end{align}
\end{corollary}

Given an $r$-graph $\mathcal{H}$ and a vertex set $U\subseteq V(\mathcal{H})$, we use $\mathcal{H}[U]$ to denote the \textbf{induced subgraph} of $\mathcal{H}$ on $U$. 
Similarly, for $r$ pairwise disjoint vertex sets $V_1, \ldots, V_r \subseteq V(\mathcal{H})$, we use $\mathcal{H}[V_1, \ldots, V_r]$ to denote the collection of edges in $\mathcal{H}$ that contain exactly one vertex from each $V_i$.  
\begin{proposition}\label{PROP:random-sample}
    Let $r\ge 2$ be an integer and $p \ge 1$ be a real number. 
    Let $\mathcal{G}$ be an $r$-graph on an $n$-set $V$ and let  $U \subseteq V$ be a vertex set. 
    For every $m \le n$, there exists a set $W \subseteq V$ of size $m$ such that the induced subgraph $\mathcal{H} \coloneqq \mathcal{G}[U\cup W]$ satisfies 
    \begin{align*}
        \sum_{v\in U}d_{\mathcal{H}}^{p}(v)
        \ge (1+o_{m}(1))\left(\frac{m}{n}\right)^{p(r-1)} \sum_{v\in U}d_{\mathcal{G}}^{p}(v).
    \end{align*}
\end{proposition}
\begin{proof}[Proof of Proposition~\ref{PROP:random-sample}]
    Choose uniformly at random an $m$-set $\mathbf{W}$ from $V$. 
    For each $v\in U$, an edge $e \in L_{\mathcal{G}}(v)$ is contained in $\mathbf{W}$ with probability  
    \begin{align*}
        \mathbb{P}\left[e\subseteq \mathbf{W}\right]
        = \frac{\binom{n-(r-1)}{m-(r-1)}}{\binom{n}{m}}
        = (1+o_{m}(1))\left(\frac{m}{n}\right)^{r-1}. 
    \end{align*}
    For every $v \in U$, let $d_{\mathcal{G}}(v,\mathbf{W}) \coloneqq \left|L_{\mathcal{G}}(v) \cap \binom{\mathbf{W}}{r-1}\right|$, noting from the equation above that $\mathbb{E}\left[d_{\mathcal{G}}(v,\mathbf{W})\right] = (1+o_{m}(1))\left(\frac{m}{n}\right)^{r-1} d_{\mathcal{G}}(v)$. Combining this with Fact~\ref{FACT:Holder} and the linearity of expectation, we obtain 
    \begin{align*}
        \mathbb{E}\left[\sum_{v\in U}d_{\mathcal{G}}^{p}(v,\mathbf{W}) \right]
        & = \sum_{v\in U} \mathbb{E}\left[d_{\mathcal{G}}^{p}(v,\mathbf{W})\right] \\
        & \ge \sum_{v\in U} \mathbb{E}\left[d_{\mathcal{G}}(v,\mathbf{W})\right]^{p} \\
        & = \sum_{v\in U} \left((1+o_{m}(1))\left(\frac{m}{n}\right)^{r-1} d_{\mathcal{G}}(v)\right)^{p} \\
        & = (1+o_{m}(1))\left(\frac{m}{n}\right)^{p(r-1)} \sum_{v\in U}d_{\mathcal{G}}^{p}(v).
    \end{align*}
    Therefore, there exists a set $W \subseteq V$ of size $m$ such that the induced subgraph $\mathcal{H} \coloneqq \mathcal{G}[U\cup W]$ satisfies 
    $\sum_{v\in U}d_{\mathcal{H}}^{p}(v) \ge \sum_{v\in U}d_{\mathcal{G}}^{p}(v,W) 
    \ge (1+o_{m}(1))\left(\frac{m}{n}\right)^{p(r-1)} \sum_{v\in U}d_{\mathcal{G}}^{p}(v)$. 
\end{proof}

\begin{theorem}[Erd{\H o}s~\cite{Erdos64}]\label{THM:Erdos-exponent}
    For every degenerate family $\mathcal{F}$ of $r$-graphs, there exists a constant $\delta > 0$ such that 
    \begin{align*}
        \mathrm{ex}(n,\mathcal{F})
        = O\left(n^{r-\delta}\right).
    \end{align*}
\end{theorem}

We say an $r$-graph $\mathcal{H}$ is \textbf{semibipartite} if there exists a bipartition $V_1 \cup V_2 = V(\mathcal{H})$ such that every edge in $\mathcal{H}$ contains exactly one vertex from $V_1$. For convenience, we write $\mathcal{H} = \mathcal{H}[V_1, V_2]$ to emphasize that $\mathcal{H}$ is semibipartite with respect to the bipartition $V_1 \cup V_2 = V(\mathcal{H})$.
Given a family $\mathcal{F}$ of $r$-graphs, we use $\mathrm{ex}(m,n,\mathcal{F})$ to denote the maximum number of edges in an $m$ by $n$ semibipartite $\mathcal{F}$-free $r$-graph. 
The function $\mathrm{ex}_p(m,n,\mathcal{F})$ is defined analogously$\colon$ for every $p\geq 1$, $\mathrm{ex}_p(m,n,\mathcal{F})$ is the maximum $p$-norm of an $m$ by $n$ semibipartite $\mathcal{F}$-free $r$-graph.
\begin{proposition}\label{PROP:r-partite-subgp-p-norm}
    Every $r$-graph $\mathcal{G}$ on $n$ vertices contains a balanced $r$-partite subgraph $\mathcal{H}$ such that 
    \begin{align*}
        \norm{\mathcal{H}}_{p}
        \ge \left(\frac{r!}{r^r} + o(1)\right)^{p}\norm{\mathcal{G}}_{p}. 
    \end{align*}
    In particular, for $r=2$ we have
    \begin{align}\label{equ:Zaran-vs-Turan}
        \mathrm{ex}_{p}(n,n,\mathcal{F})
        \ge \left(\frac{1}{2} + o(1)\right)^{p}\mathrm{ex}_{p}(2n,\mathcal{F}).
    \end{align}
\end{proposition}
\begin{proof}[Proof of Proposition~\ref{PROP:r-partite-subgp-p-norm}]
    Choose a balanced $r$-partition $V_1 \cup \cdots \cup V_{r} = V(\mathcal{G})$ uniformly at random. More specifically,  we first fix integers $m_1, \ldots, m_{r}$ satisfying $m_r+ 1\ge m_1 \ge \cdots \ge m_r$ and $m_1 + \cdots+ m_r = n$. Then we inductively select uniformly at random an $m_i$-set $V_i$ from $V(\mathcal{G})\setminus (V_{1} \cup \cdots \cup V_{i-1})$, where $V_0 \coloneqq  \emptyset$. 
    Let $\mathbf{G} \coloneqq \mathcal{G}[V_1, \ldots, V_r]$ and $V \coloneqq V(\mathcal{G})$. 
    Similarly to the proof of Proposition~\ref{PROP:random-sample}, it follows from Fact~\ref{FACT:Holder} and the linearity of expectation that  
    \begin{align*}
        \mathbb{E}\left[\norm{\mathbf{G}}_{p}\right]
        = \sum_{v\in V} \mathbb{E}\left[d_{\mathbf{G}}^{p}(v)\right] 
        & \ge \sum_{v\in V} \left(\mathbb{E}\left[d_{\mathbf{G}}(v)\right]\right)^{p}  \\
        & = \sum_{v\in V}\left(\sum_{e\in \mathcal{G}\colon v\in e}\mathbb{P}\left[e\in \mathbf{G}\right]\right)^{p} \\
        & = \sum_{v\in V}\left(\left(\frac{r!}{r^r}+ o(1)\right) \cdot  d_{\mathcal{G}}(v)\right)^{p} 
         = \left(\frac{r!}{r^r}+ o(1)\right)^{p} \norm{\mathcal{G}}_{p}. 
    \end{align*}
    Therefore, there exists a balanced $r$-partition $V_1 \cup \cdots \cup V_{r} = V(\mathcal{G})$ such that the $r$-partite subgraph $\mathcal{H} \coloneqq \mathcal{G}[V_1, \ldots, V_r]$ satisfies $\norm{\mathcal{H}}_{p}
        \ge \left(\frac{r!}{r^r} + o(1)\right)^{p}\norm{\mathcal{G}}_{p}$. 
\end{proof}

Let $K_{s_1, \ldots, s_r}^{r}$ be the complete $r$-partite $r$-graph with parts of sizes $s_1, \ldots, s_r$, respectively.
Extending classical theorems of K{\H o}v\'{a}ri--S\'{o}s--Tur\'{a}n~\cite{KST54} and Erd\H{o}s~\cite{Erdos64}, the following upper bound for $\mathrm{ex}(m,n,K_{s_1, \ldots, s_r}^{r})$ was proved in~\cite{HHLLYZ23}.

\begin{proposition}[{\cite[Proposition~2.1]{HHLLYZ23}}]\label{PROP:hypergraph-KST-Zaran}
    Suppose that $r \ge 3$, $s_r\ge \cdots \ge s_1 \ge 1$, and $m, n\ge 1$ are integers.  
    Then 
    \begin{align*}
        \mathrm{ex}(m,n,K_{s_1, \ldots, s_r}^{r})
        \le \frac{(s_2+\cdots+s_r-r+1)^{\frac{1}{s_1}}}{r-1}mn^{r-1-\frac{1}{s_1\cdots s_{r-1}}} + (s_1-1)\binom{n}{r-1}. 
    \end{align*}    
\end{proposition}
\begin{proposition}\label{PROP:hypergraph-KST-Zaran-b}
    Let $r\ge 2$ be an integer and $\mathcal{F}$ be a degenerate family of $r$-graphs. Suppose that $\mathrm{ex}(n,\mathcal{F}) = O(n^{1+\alpha})$ for some constant $\alpha$. 
    Then there exist constants $C_{\mathcal{F}}, N_0$ such that
    \begin{align*}
        \mathrm{ex}(m,n,\mathcal{F})
        \le C_{\mathcal{F}} m^{1+\alpha-(r-1)} n^{r-1}
        \quad\text{for all}\quad 
        n \ge m \ge N_0.
    \end{align*}    
\end{proposition}
\begin{proof}[Proof of Proposition~\ref{PROP:hypergraph-KST-Zaran-b}]
    Let $C, N_0$ be constants such that $\mathrm{ex}(n,\mathcal{F}) \le C n^{1+\alpha}$ for every $n \ge N_0$. 
    Let $C_{\mathcal{F}} \coloneqq 2^{2+\alpha} C$.
    Suppose to the contrary that there exists an $\mathcal{F}$-free $m$ by $n$ semibipartite $r$-graph $\mathcal{G} = \mathcal{G}[V_1, V_2]$ with $|\mathcal{G}| > C_{\mathcal{F}} m^{1+\alpha-(r-1)} n^{r-1}$, where $n \ge m \ge N_0$. 
    Similar to the proof of Proposition~\ref{PROP:random-sample}, there exists a set $U \subseteq V_2$ of size $m$ such that the induced subgraph $\mathcal{H} \coloneqq \mathcal{G}[V_1\cup U]$ satisfies 
    \begin{align*}
        |\mathcal{H}|
        = \sum_{v\in V_1} d_{\mathcal{H}}(v) 
        & \ge (1+o(1))\left(\frac{m}{n}\right)^{r-1} \sum_{v\in V_1} d_{\mathcal{G}}(v) \\
        & \ge \frac{1}{2} \left(\frac{m}{n}\right)^{r-1} |\mathcal{G}| \\
        & > \frac{1}{2} \left(\frac{m}{n}\right)^{r-1} C_{\mathcal{F}} m^{1+\alpha-(r-1)} n^{r-1} 
         = C (2m)^{1+\alpha}
         \ge \mathrm{ex}\left(|V_1 \cup U|, \mathcal{F}\right), 
    \end{align*}
    a contradiction.     
\end{proof}
\section{Regularization under the $p$-norm}\label{SEC:regular-lemma}
In this section, we prove the following extension of the classical $\Delta$-almost-Regularization Theorem by Erd{\H o}s--Simonovits (see e.g.~{\cite[Theorem~2.19]{FS13}}).
\begin{lemma}\label{LEMMA:regularization-p-norm}
    Let $r \ge 2$ be an integer. 
    Let $\alpha \in (r-2, r-1)$, $p \in \left[1, \frac{1}{r-1-\alpha}\right)$, and $C > 0$ be real numbers. 
    Then for every $\varepsilon\in (0,1)$, there exist constants $K$ and $N_0$ such that the following holds for every $n \ge N_0$. 
    Suppose $\G$ is an $r$-graph on $n$ vertices with $\norm{\G}_{p} \ge C  n^{1+p \alpha}$. 
    Then $\mathcal{G}$ contains a subgraph $\mathcal{H}$ on $m$ vertices satisfying 
    \begin{enumerate}[label=(\roman*)]
        \item\label{LEMMA:regularization-p-norm-1} $\norm{\mathcal{H}}_p \ge (1-\varepsilon)C  m^{1+p \alpha}$,
        \item\label{LEMMA:regularization-p-norm-2} $m\ge \frac{\left(C^{1/\delta} n\right)^{\frac{\delta}{1-3\delta}}}{2}$, where  $\delta \coloneqq \frac{1 - p(r-1-\alpha)}{4}$, 
        \item\label{LEMMA:regularization-p-norm-3} $\Delta(\h) \le \left(\frac{K}{1-\varepsilon}\cdot \frac{\norm{\mathcal{H}}_p}{m} \right)^{1/p}$, and 
        \item\label{LEMMA:regularization-p-norm-4} $|\mathcal{H}|> \hat{C} m^{1+\alpha}$, where $\hat{C} \coloneqq
        \frac{(1-\varepsilon) C^{1/p}}{r K^{\frac{p-1}{p}}}$.
    \end{enumerate}
\end{lemma}
\begin{proof}[Proof of Lemma~\ref{LEMMA:regularization-p-norm}]
    Let $r \ge 2$, $\alpha \in (r-2, r-1)$, $p \in \left[1, \frac{1}{r-1-\alpha}\right)$, and $C > 0$ be as assumed in Lemma~\ref{LEMMA:regularization-p-norm}. Since $p \in \left[1, \frac{1}{r-1-\alpha}\right)$, the constant $\delta = \frac{1 - p(r-1-\alpha)}{4}$ satisfies $0< \delta < 1/4$. 
    Fix $\varepsilon \in (0,1)$. 
    Let $\varepsilon_{1}$ be the real number in $(0,\varepsilon)$ such that 
    \begin{align*}
        1-\varepsilon_{1}-((r-1)\varepsilon_{1})^{1/p} 
        = (1-\varepsilon)^{1/p}.
    \end{align*}
    Let $K$ be a constant satisfying 
    \begin{align*}
        K^{\delta} \ge 2^{1+p(r-1)}
        \quad\text{and}\quad 
        \frac{K^{1+p \alpha}}{K^{p(r-1)}}\cdot \frac{\varepsilon_{1}}{2^{2+p \alpha}} =K^{4\delta} \cdot \frac{\varepsilon_{1}}{2^{2+p \alpha}} > K^{2\delta}. 
    \end{align*}
    Let $N_1$ be the constant such that Proposition~\ref{PROP:random-sample} holds with $o_{m}(1) \ge -1/2$ for all $m \ge N_1$. 
    Let $N_0 \gg N_1$ be a sufficiently large integer and  $\G$ be an $r$-graph on $n \ge N_0$ vertices with $\norm{\G}_{p} \ge C  n^{1+p \alpha}$.
    
    For convenience, for every $r$-graph $\mathcal{K}$, we define  
    \begin{align*}
        \Phi(\mathcal{K}) 
        \coloneqq \frac{\norm{\mathcal{K}}_{p}}{|V(\mathcal{K})|^{1+p\alpha}}.
    \end{align*}
    Note that $\Phi(\mathcal{G}) \ge C$.

    We will define a sequence of subgraphs $\G_0 = \mathcal{G} \supseteq \G_1 \supseteq \cdots \supseteq \G_k$ for some $k \ge 0$ such that 
    \begin{align*}
        \Phi(\mathcal{G}_{i+1})
        \ge K^{2\delta } \cdot  \Phi(\mathcal{G}_{i})
        \ge \Phi(\mathcal{G}_{i})
        \quad\text{and}\quad 
        \left(\frac{1}{K}\right)^{i+1}n 
        \le |V(\G_{i+1})| 
        \le \left(\frac{2}{K}\right)^{i+1}n
    \end{align*}
    for every $i \in [0,k-1]$.
    
    Suppose we have defined $\G_i$ for some $i\ge 0$.
    %
    Let 
    \begin{align*}
        U_i 
        \coloneqq \left\{v\in V(\G_i) \colon d_{\mathcal{G}_{i}}^{p}(v) \ge \frac{K \cdot \norm{\G_i}_p}{|V(\mathcal{G}_{i})|} \right\}. 
    \end{align*}
    It follows from 
    \begin{align*}
       \norm{\G_i}_p 
       = \sum_{v\in V(\G_i)} d_{\mathcal{G}_{i}}^p(v) 
       \ge  \sum_{v\in U_i} d_{\mathcal{G}_{i}}^p(v) 
       \ge |U_i| \cdot \frac{K \cdot \norm{\G_i}_p }{|V(\mathcal{G}_{i})|}
    \end{align*}
    that 
    \begin{align*}
        |U_i| \le \frac{|V(\mathcal{G}_{i})|}{K}. 
    \end{align*}
    If $\sum_{v\in U_i} d_{\mathcal{G}_{i}}^p(v)< \varepsilon_{1} \norm{\G_i}_p$ or $|V(\mathcal{G}_{i})| \le N_1$, then we stop the process and set $k \coloneqq i$. 
    Otherwise, we apply Proposition~\ref{PROP:random-sample} to $\mathcal{G}_{i}$ with $U$ and $m$ in the proposition corresponding to $U_{i}$ and $|V(\mathcal{G}_{i})|/K$ here. Let $V_{i+1} \subseteq V(\mathcal{G}_{i})$ be the $\frac{|V(\mathcal{G}_{i})|}{K}$-set returned by Proposition~\ref{PROP:random-sample}, and let $\mathcal{G}_{i+1} \coloneqq \mathcal{G}_{i}[U_{i} \cup V_{i+1}]$.
    By Proposition~\ref{PROP:random-sample}, we have  
    \begin{align*}
        \norm{\mathcal{G}_{i+1}}_{p}
        & \ge \sum_{v\in U_{i}}d_{\mathcal{G}_{i+1}}^{p}(v)  \\
        & \ge (1+o(1)) \left(\frac{|V(\mathcal{G}_{i})|/K}{|V(\mathcal{G}_{i})|}\right)^{p(r-1)} \sum_{v\in U_{i}}d_{\mathcal{G}_{i}}^{p}(v) 
         \ge \frac{\sum_{v\in U_{i}}d_{\mathcal{G}_{i}}^{p}(v)}{2K^{p(r-1)}}
        \ge \frac{\varepsilon_1  \norm{\mathcal{G}_{i}}_{p}}{2K^{p(r-1)}}.
    \end{align*}
    It follows that 
    \begin{align}\label{equ:regularization-Psi-increase}
        \Phi(\mathcal{G}_{i+1})
        = \frac{\norm{\G_{i+1}}_p}{|V(\mathcal{G}_{i+1})|^{1+p\alpha}} 
        & \ge \frac{\varepsilon_1  \norm{\mathcal{G}_{i}}_{p}}{2K^{p(r-1)}}/\left(\frac{2|V(\mathcal{G}_{i})|}{K}\right)^{1+p\alpha} \notag \\
        & = \frac{\varepsilon_1}{2K^{p(r-1)}} \cdot \frac{K^{1+p\alpha}}{2^{1+p\alpha}} \cdot \frac{\norm{\mathcal{G}_{i}}_{p}}{|V(\mathcal{G}_{i})|^{1+p\alpha}}
        > K^{2\delta }\cdot \Phi(\mathcal{G}_{i}).
    \end{align}
    Additionally, it follows from the inductive hypothesis that 
    \begin{align}
        |V(\mathcal{G}_{i+1})| 
        & = |U_i \cup V_{i+1}|
        \ge |V_{i+1}|
        \ge \frac{|V(\mathcal{G}_{i})|}{K}
        \ge \left(\frac{1}{K}\right)^{i+1}n, 
        \quad\text{and}\quad \label{equ:regularization-V-Gi-1}\\
        |V(\mathcal{G}_{i+1})|
        & = |U_i \cup V_{i+1}|
        \le  |U_i| + |V_{i+1}| 
        \le \frac{|V(\mathcal{G}_{i})|}{K} + \frac{|V(\mathcal{G}_{i})|}{K}
        \le \left(\frac{2}{K}\right)^{i+1}n, \label{equ:regularization-V-Gi-2}
    \end{align}
    as desired. 

    We claim that the process defined above stops after at most $k_{\ast} \coloneqq \log_{K}(n/N_1)$ steps.  
    Indeed, suppose this is not true. Then at the $k_{\ast}$-step, by~\eqref{equ:regularization-Psi-increase}, we would have 
    \begin{align}\label{equ:regularization-k-ast}
        \frac{\norm{\mathcal{G}_{k_{\ast}}}}{|V(\mathcal{G}_{k_{\ast}})|^{1+p\alpha}}
        = \Phi(\mathcal{G}_{k_{\ast}})
        \ge \left(K^{2\delta}\right)^{k_{\ast}} \cdot \Phi(\mathcal{G}_0)
        \ge C K^{2\delta k_{\ast}}. 
    \end{align}
    It is trivially true that 
    \begin{align*}
        \frac{\norm{\mathcal{G}_{k_{\ast}}}}{|V(\mathcal{G}_{k_{\ast}})|^{1+p\alpha}}
        = \frac{\sum_{v\in V(\mathcal{G}_{k_{\ast}})}d_{\mathcal{G}_{k_{\ast}}}^{p}(v)}{|V(\mathcal{G}_{k_{\ast}})|^{1+p\alpha}}
        & \le \frac{|V(\mathcal{G}_{k\ast})| \cdot |V(\mathcal{G}_{k\ast})|^{p(r-1)}}{|V(\mathcal{G}_{k_{\ast}})|^{1+p\alpha}} \\
        & = |V(\mathcal{G}_{k_{\ast}})|^{p(r-1-\alpha)}
        \le |V(\mathcal{G}_{k_{\ast}})|.
    \end{align*}
    Combining this with~\eqref{equ:regularization-V-Gi-2} and~\eqref{equ:regularization-k-ast}, we obtain 
    \begin{align*}
        C K^{2\delta k_{\ast}}
        \le |V(\mathcal{G}_{k_{\ast}})|
        \le \left(\frac{2}{K}\right)^{k_{\ast}} n.
    \end{align*}
    It follows that 
    \begin{align*}
        n 
        \ge C K^{2\delta k_{\ast}} \left(\frac{K}{2}\right)^{k_{\ast}}
        = C \left(\frac{K^{2\delta}}{2}\right)^{k_{\ast}} K^{k_{\ast}}
        \ge C K^{\delta k_{\ast}} K^{k_{\ast}}
        = C \left(\frac{n}{N_{1}}\right)^{1+\delta},
    \end{align*}
    which is a contradiction since $C, \delta, N_1 > 0$ are fixed and $n$ is sufficiently large. 
    Therefore, the process defined above stops after at most $k_{\ast} \coloneqq \log_{K}(n/N_1)$ steps. 

    Recalling that $k$ is the final step of the process, and using ~\eqref{equ:regularization-V-Gi-1}, we have 
    \begin{align*}
        |V(\mathcal{G}_{k})|
        \ge \left(\frac{1}{K}\right)^{k} n
        \ge \left(\frac{1}{K}\right)^{k_{\ast}} n
        \ge N_1. 
    \end{align*}
    This means that the process stopped due to 
    \begin{align}\label{equ:Regularization-Uk-sum}
              \sum_{v\in U_{k}}d_{\mathcal{G}_{k}}^{p}(v) < \varepsilon_1 \norm{\mathcal{G}_{k}}_{p}. 
    \end{align}
    Let $\h$ denote the induced subgraph of $\mathcal{G}_{k}$ on $W \coloneqq V(\mathcal{G}_{k})\setminus U_k$ and let $m \coloneqq |W|$. Recall that  
    \begin{align*}
        U_k 
        \coloneqq \left\{v\in V(\mathcal{G}_{k}) \colon d_{\mathcal{G}_{k}}^{p}(v) \ge \frac{K \cdot \norm{\mathcal{G}_{k}}_p}{|V(\mathcal{G}_{k})|} \right\}
    \end{align*}
    We will show that $\mathcal{H}$ satisfies the assertions in Lemma~\ref{LEMMA:regularization-p-norm}.

    Let $\mathcal{R} \coloneqq \mathcal{G}_{k}\setminus \mathcal{H}$. 
    Note that every edge in $\mathcal{R}$ contains at least one vertex from $U_{k}$. 
    Therefore,
    \begin{align*}
        \sum_{v\in W}d_{\mathcal{R}}(v)
        \le (r-1)\cdot |\mathcal{R}|
        \le (r-1) \cdot \sum_{v\in U_{k}}d_{\mathcal{R}}(v). 
    \end{align*}
    Since $\mathcal{R} \subseteq \mathcal{G}_{k}$, it follows from the definition of $U_{k}$ and~\eqref{equ:Regularization-Uk-sum} that 
    \begin{align*}
        \sum_{v\in W}d_{\mathcal{R}}^{p}(v)
        \le \sum_{v\in W}d_{\mathcal{R}}(v) \cdot d_{\mathcal{G}_{k}}^{p-1}(v)
        & \le \sum_{v\in W}d_{\mathcal{R}}(v) \cdot \left(\frac{K \cdot \norm{\mathcal{G}_{k}}_p}{|V(\mathcal{G}_{k})|}\right)^{\frac{p-1}{p}} \\
        & \le (r-1) \cdot \sum_{v\in U_{k}}d_{\mathcal{R}}(v) \cdot \left(\frac{K \cdot \norm{\mathcal{G}_{k}}_p}{|V(\mathcal{G}_{k})|}\right)^{\frac{p-1}{p}} \\ 
        & \le (r-1) \cdot \sum_{v\in U_{k}}d_{\mathcal{G}_{k}}(v) \cdot d_{\mathcal{G}_{k}}^{p-1}(v) \\
        & = (r-1) \cdot \sum_{v\in U_{k}}d_{\mathcal{G}_{k}}^{p}(v)
        < (r-1)\varepsilon_1 \norm{\mathcal{G}_{k}}_{p}.
    \end{align*}
    If, for the sake of contradiction, it holds that $\sum_{v\in W}d^p_{\mathcal{H}}(v) = \norm{\mathcal{H}}_p < (1-\varepsilon)\norm{\G_k}_p$, then it follows from the inequality above that
    \begin{align*}
        \sum_{v\in W}d_{\mathcal{G}_{k}}^p(v)
        & = \sum_{v\in W}\left(d_{\mathcal{H}}(v) + d_{\mathcal{R}}(v)\right)^p  \\
        & \le \left(\left(\sum_{v\in W}d^p_{\mathcal{H}}(v)\right)^{1/p} + \left(\sum_{v\in W} d_{\mathcal{R}}^{p}(v)\right)^{1/p}\right)^p \\
        & \le \left(\norm{\mathcal{H}}_{p}^{1/p} + \left((r-1)\varepsilon_1 \norm{\mathcal{G}_{k}}_{p}\right)^{1/p}\right)^p  \\
        & < \left(\left((1-\varepsilon)\norm{\G_k}_p\right)^{1/p} + \left((r-1)\varepsilon_1 \norm{\mathcal{G}_{k}}_{p}\right)^{1/p}\right)^p \\
        & = \left((1-\varepsilon)^{1/p} + (r-1)^{1/p}\varepsilon_1^{1/p}\right)^{p} \norm{\mathcal{G}_{k}}_{p} 
         = (1-\varepsilon_{1})^{p} \norm{\mathcal{G}_{k}}_{p}, 
    \end{align*}
    where the first inequality follows from  Fact~\ref{FACT:Minkowski} and
    the last equality follows from the definition of $\varepsilon_1$. 
    Combining this with~\eqref{equ:Regularization-Uk-sum}, we obtain 
    \begin{align*}
        \norm{\mathcal{G}_{k}}_p
         =\sum_{v\in U_k}d_{\mathcal{G}_{k}}^p(v) + \sum_{v\in W}d_{\mathcal{G}_{k}}^p(v) \notag 
         < \varepsilon_1 \norm{\mathcal{G}_{k}}_{p} + (1-\varepsilon_{1})^{p} \norm{\mathcal{G}_{k}}_{p}
        \leq \norm{\mathcal{G}_{k}}_{p},
    \end{align*}
    a contradiction. 
    Therefore, we have 
    \begin{align}\label{equ:Regulization-Hp-Gk}
        \norm{\mathcal{H}}_p \ge (1-\varepsilon)\norm{\G_k}_p, 
    \end{align}
    which implies that 
    \begin{align*}
        \Phi(\mathcal{H})
        = \frac{\norm{\mathcal{H}}_p}{m^{1+p\alpha}}
        \ge \frac{(1-\varepsilon)\norm{\G_k}_p}{|V(\mathcal{G}_{k})|^{1+p\alpha}} 
        \ge (1-\varepsilon) \frac{\norm{\G_0}_p}{|V(\mathcal{G}_{0})|^{1+p\alpha}} 
        \ge (1-\varepsilon) C .
    \end{align*}
    Here, we used the fact that 
    \begin{align*}
        \Phi(\mathcal{G}_{k}) 
        \ge K^{2\delta } \cdot \Phi(\mathcal{G}_{k-1}) 
        \ge \cdots \ge K^{2k\delta} \cdot \Phi(\mathcal{G}_{0}) 
        \ge \Phi(\mathcal{G}_{0}) 
        \ge C.
    \end{align*}
    This completes the proof of Lemma~\ref{LEMMA:regularization-p-norm}~\ref{LEMMA:regularization-p-norm-1}.

    Next, we prove Lemma~\ref{LEMMA:regularization-p-norm}~\ref{LEMMA:regularization-p-norm-2}. 
    Note that by $W = V(\mathcal{G}_k)\setminus|U_k|$ and $|U_k| \le \frac{|V(\mathcal{G}_k)|}{K}$, we have $|W| \ge |V(\mathcal{G}_k)|- \frac{|V(\mathcal{G}_k)|}{K}$.
    Recall the following results that we have established$\colon$ 
    \begin{claim}\label{CLAIM:regularization-k}
        We have 
        \begin{enumerate}[label=(\roman*)]
            \item\label{CLAIM:regularization-k-1} $|V(\mathcal{G}_k)|- \frac{|V(\mathcal{G}_k)|}{K} \le |W| = m \le |V(\mathcal{G}_{k})|$.
            \item\label{CLAIM:regularization-k-2} $\left(\frac{1}{K}\right)^{k}n 
            \le |V(\mathcal{G}_k)| 
            \le \left(\frac{2}{K}\right)^{k}n$.
            \item\label{CLAIM:regularization-k-3} $K^{2k\delta}C  \le \frac{\norm{\mathcal{G}_{k}}_{p}}{|V(\mathcal{G}_k)|^{1+p\alpha}}  \le  \frac{|V(\mathcal{G}_k)|^{1+p(r-1)}}{|V(\mathcal{G}_k)|^{1+p\alpha}} = |V(\mathcal{G}_k)|^{p(r-1-\alpha)}
            = |V(\mathcal{G}_k)|^{1-4\delta}$.
        \end{enumerate}
    \end{claim}
    It follows from Claim~\ref{CLAIM:regularization-k}~\ref{CLAIM:regularization-k-2} and~\ref{CLAIM:regularization-k-3} that 
    \begin{align*}
        K^{2k\delta}C  
        \le|V(\mathcal{G}_k)|^{1-4\delta }
        \le \left(\left(\frac{2}{K}\right)^{k}n\right)^{1-4\delta }. 
    \end{align*}
    Since $K^{\delta} \ge 2^{1+p(r-1)} \ge 2^{p(r-1-\alpha)} = 2^{1-4\delta}$, the inequality above implies that 
    \begin{align*}
        n^{1-4\delta }
        \ge \frac{K^{2k\delta } C  K^{k(1-4\delta )}}{2^{k(1-4\delta )}} 
        \ge \frac{K^{2k\delta} C  K^{k(1-4\delta )}}{K^{k\delta}}
        = K^{k\left(1-3\delta\right)} C. 
    \end{align*}
    It follows that $K^k 
        \le n^{\frac{1-4\delta }{1-3\delta}} C^{-\frac{1}{1-3\delta}}$. 
    %
    Therefore, 
    \begin{align*}
        m 
        \ge |V(\mathcal{G}_k)|- \frac{|V(\mathcal{G}_k)|}{K}
        \ge \frac{|V(\mathcal{G}_k)|}{2}
        \ge \frac{1}{2} \left(\frac{1}{K}\right)^{k}n
        \ge \frac{C^{\frac{1}{1-3\delta}} n^{\frac{\delta}{1-3\delta}}}{2},
    \end{align*}
    proving Lemma~\ref{LEMMA:regularization-p-norm}~\ref{LEMMA:regularization-p-norm-2}. 
    
    It follows from the definition of $U_k$ and  $\norm{\mathcal{H}}_p \ge (1-\varepsilon)\norm{\G_k}_p$ (by~\eqref{equ:Regulization-Hp-Gk}) that 
    \begin{align}\label{equ:reg-Delta}
        \Delta(\h) 
        &<\left(\frac{K \cdot \norm{\G_k}_p}{|V(\mathcal{G}_{k})|}\right)^{1/p}
        \le \left(\frac{K \cdot \norm{\mathcal{H}}_p}{(1-\varepsilon) |V(\mathcal{G}_{k})|}\right)^{1/p}
        \le \left(\frac{K \cdot \norm{\mathcal{H}}_p}{(1-\varepsilon) m}\right)^{1/p}. 
    \end{align}
    This proves Lemma~\ref{LEMMA:regularization-p-norm}~\ref{LEMMA:regularization-p-norm-3}.

    Finally, it follows from 
    \begin{align*}
        \norm{\mathcal{H}}_p
        = \sum_{v\in W} d_{\mathcal{H}}^{p}(v)
        \le \sum_{v\in W} d_{\mathcal{H}}(v) \cdot \left(\Delta(\mathcal{H})\right)^{p-1}
        = r \cdot |\mathcal{H}| \cdot \left(\Delta(\mathcal{H})\right)^{p-1}
    \end{align*}
    and~\eqref{equ:reg-Delta} that 
    \begin{align*}
        |\mathcal{H}|
        \ge \frac{\norm{\mathcal{H}}_p}{r \left(\Delta(\mathcal{H})\right)^{p-1}}
        > \frac{\norm{\mathcal{H}}_p}{r \left(\frac{K \cdot \norm{\mathcal{H}}_p}{(1-\varepsilon) m}\right)^{\frac{p-1}{p}}}
        = \frac{1}{r}\left(\frac{(1-\varepsilon)m}{K}\right)^{\frac{p-1}{p}}  \norm{\mathcal{H}}_p^{1/p}.
    \end{align*}
    Combining this with $\norm{\mathcal{H}}_p \ge (1-\varepsilon)\norm{\G_k}_{p} \ge (1-\varepsilon) C  m^{1+p \alpha}$ (by~\eqref{equ:Regulization-Hp-Gk} and Claim~\ref{CLAIM:regularization-k}~\ref{CLAIM:regularization-k-3}), we obtain 
    \begin{align*}
        |\mathcal{H}|
        >  \frac{1}{r}\left(\frac{(1-\varepsilon)m}{K}\right)^{\frac{p-1}{p}}  \left((1-\varepsilon)C  m^{1+p \alpha}\right)^{1/p}
        =  \frac{(1-\varepsilon) C^{1/p} m^{1+\alpha}}{r K^{\frac{p-1}{p}}},   
    \end{align*}
    which proves Lemma~\ref{LEMMA:regularization-p-norm}~\ref{LEMMA:regularization-p-norm-4}.
\end{proof}

\section{Proof of Theorem~\ref{THM:main-r-graph}}\label{SEC:proof-r-gp}
In this section, we prove Theorem~\ref{THM:main-r-graph}. 
This will be achieved through the following two propositions, first addressing the case $p < \frac{1}{r-1-\alpha}$.
\begin{proposition}\label{PROP:r-graph-small-p}
    Let $r \ge 2$ be an integer. 
    Suppose that $\mathcal{F}$ is a degenerate family of $r$-graphs satisfying $\mathrm{ex}(n, \mathcal{F}) = O(n^{1+\alpha})$ for some constant $\alpha$. 
    Then for every $p\in \left(1, \frac{1}{r-1-\alpha} \right)$, there exists a constant $C_{\mathcal{F}}$ such that for all sufficiently large $n$, 
    \begin{align*}
        \mathrm{ex}_p(n, \mathcal{F}) 
        \le C_{\mathcal{F}} \cdot n^{1+p\alpha}. 
    \end{align*}
    In particular, if $\mathrm{ex}(n, \mathcal{F}) = \Theta(n^{1+\alpha})$, then together with \eqref{equ:Turan-number-p-lower-bound}, 
    \begin{align*}
        \mathrm{ex}_p(n, \mathcal{F}) = \Theta(n^{1+p\alpha})
        \quad\text{for every}\quad 
        p\in \left(1, \frac{1}{r-1-\alpha} \right).
    \end{align*}
\end{proposition}
\begin{proof}[Proof of Proposition~\ref{PROP:r-graph-small-p}]
    Let $C, N_0>0$ be constants such that $\mathrm{ex}(n, \mathcal{F}) \le C n^{1+\alpha}$ for every $n \ge N_0$. 
    Let $\delta \coloneqq \frac{1-(r-1-\alpha)p}{4} \in \left(0,\frac{1}{4}\right)$.
    Fix $\varepsilon \coloneqq \frac{1}{2}$ and let $K = K(r,\alpha, p, \varepsilon)$ be the constant returned by Lemma~\ref{LEMMA:regularization-p-norm}. 
    Let $C_{\mathcal{F}} \coloneqq 2^p r^p K^{p-1} C^p$ and $N_1 \coloneqq {(2N_0)^{\frac{1-3\delta}{\delta}}}/{C_{\mathcal{F}}^{1/\delta}}$. 
    
    Suppose to the contrary that there exists an $\mathcal{F}$-free $r$-graph $\mathcal{G}$ on $n \ge N_1$ vertices with $\norm{\mathcal{G}}_{p} > C_{\mathcal{F}}\cdot n^{1+p\alpha}$. 
    Then, by Lemma~\ref{LEMMA:regularization-p-norm}, there exists a subgraph $\mathcal{H} \subseteq \mathcal{G}$ on $m \ge {\left(C_{\mathcal{F}}^{1/\delta} n\right)^{\frac{\delta}{1-3\delta}}}/{2} \ge N_0$ vertices with $|\mathcal{H}| > {(1-\varepsilon)C_{\mathcal{F}}^{1/p}m^{1+\alpha}}/{\left(r K^{\frac{p-1}{p}}\right)}  = C m^{1+\alpha}$, contradicting the $\mathcal{F}$-freeness of $\mathcal{H} \subseteq \mathcal{G}$. 
\end{proof}

Next, we consider the case $p > \frac{1}{r-1-\alpha}$.
\begin{proposition}\label{PROP:r-graph-large-p}
    Let $r \ge 2$ be an integer. 
    Suppose that  $\mathcal{F}$ is a degenerate family of $r$-graphs satisfying $\mathrm{ex}(n, \mathcal{F}) = O(n^{1+\alpha})$ for some constant $\alpha \in (r-2,r-1)$. 
    Then for every $p > \frac{1}{r-1-\alpha}$, 
    \begin{align*}
        \mathrm{ex}_p(n, \mathcal{F}) 
        \le \left(\tau_{\mathrm{part}}(\mathcal{F})-1+o(1)\right)\binom{n}{r-1}^p.
    \end{align*}
\end{proposition}
\begin{proof}[Proof of Proposition~\ref{PROP:r-graph-large-p}]
    Let $F\in \mathcal{F}$ be an $r$-partite $r$-graph satisfying $\tau_{\mathrm{part}}(F) =\tau_{\mathrm{part}}(\mathcal{F})$.
    Let $A_1 \cup \cdots \cup A_r = V(F)$ be an $r$-partition of $F$ with $|A_1| \le \cdots \le |A_{r}|$ and $|A_1| = \tau_{\mathrm{part}}(F)$. 
    Let $s_i \coloneqq |A_i|$ for $i\in [r]$. 
    Note that $s_1 = |A_1| = \tau_{\mathrm{part}}(F)=\tau_{\mathrm{part}}(\mathcal{F})$. 
    Let the $(r-1)$-partite $(r-1)$-graph $F_1$ on $A_2 \cup \cdots \cup A_r$ be defined as 
    \begin{align*}
        F_1
        \coloneqq \bigcup_{v\in A_1}L_{F}(v).
    \end{align*}
    By Theorem~\ref{THM:Erdos-exponent}, there exist constants $\delta>0$ and $C>0$ such that $\mathrm{ex}(n, F_1) \le C n^{r-1-\delta }$ for every $n\in \mathbb{N}$. By reducing $\delta$, we may assume that $\delta \le \min\left\{\frac{1}{s_1\cdots s_{r-1}},~\frac{1}{2} \right\}$. 
    Let $p_{\ast} \coloneqq \frac{1}{r-1-\alpha}$.
    Let $\delta_1  >0$ be a sufficiently small constant such that, in particular,  
    \begin{align*}
        \delta_1  
        <\min\left\{ p - p_{\ast},~\frac{2+\alpha-r}{p-p_{\ast}} \right\}
        \quad\text{and}\quad
        \delta_2 
        \coloneqq 
        \delta_1  + \frac{\delta_1 }{r-1-\alpha}
        \le \min\left\{\frac{\delta}{s_1},~ \frac{p-1}{p}\right\}.
    \end{align*}
    Fix an arbitrary small constant $\varepsilon > 0$.
    Let $n$ be sufficiently large. 
    Suppose to the contrary that there exists an $n$-vertex $\mathcal{F}$-free $r$-graph $\mathcal{H}$ with 
    \begin{align*}
        \norm{\mathcal{H}}_{p}
        \ge \left(\tau_{\mathrm{part}}(\mathcal{F})-1+\varepsilon\right)\binom{n}{r-1}^p
        =\left(s_1-1+\varepsilon\right)\binom{n}{r-1}^p.
    \end{align*}
    Let 
    \begin{align*}
        V\coloneqq V(\mathcal{H}),
        \quad
        U 
        \coloneqq \left\{v\in V \colon d_{\mathcal{H}}(v) \ge n^{r-1-\delta_1 }\right\},
        \quad 
        V_1 
        \coloneqq V\setminus U, 
        \quad\text{and}\quad 
        \mathcal{H}_{1}  \coloneqq \mathcal{H}[V_{1}].
    \end{align*}
    \begin{claim}\label{CLAIM:r-gp-p-large-U-upper}
        We have $|U| \le n^{\delta_2 }$.
    \end{claim}
    \begin{proof}[Proof of Claim~\ref{CLAIM:r-gp-p-large-U-upper}]
        Suppose to the contrary that this is not true. 
        Let $U' \subseteq U$ be a set of size $n^{\delta_2 }$. 
        By Proposition~\ref{PROP:random-sample}, there exists a set $V'\subseteq V$ of size $n^{\delta_2}$ such that the induced subgraph $\mathcal{H}[U'\cup V']$ satisfies 
        \begin{align*}
            |\mathcal{H}[U'\cup V']|
            \ge \frac{1}{r}\sum_{v\in U'}d_{\mathcal{H}[U'\cup V']}(v)
            & \ge \frac{1}{r}\left(1+o(1)\right)\left(\frac{n^{\delta_2}}{n}\right)^{r-1}\sum_{v\in U'}d_{\mathcal{H}}(v) \\
            & \ge \frac{1}{2r} n^{-(r-1)(1-\delta_2)} \cdot n^{\delta_2} \cdot n^{r-1-\delta_1} 
            = \frac{n^{r\delta_2 - \delta_1}}{2r}.
        \end{align*}
        Since $r\delta_2 - \delta_1 - \delta_2 (1+\alpha) = \delta_2 (r-1-\alpha) -\delta_1 = (r-1-\alpha) \delta_1$, we have 
        \begin{align*}
            |\mathcal{H}[U'\cup V']| 
            \ge \frac{n^{r\delta_2 - \delta_1}}{2r} 
            = \frac{n^{(r-1-\alpha)\delta_1}}{2^{2+\alpha} r} \cdot \left(2n^{\delta_2}\right)^{1+\alpha}
            > \mathrm{ex}(2n^{\delta_2}, \mathcal{F})
            \ge \mathrm{ex}\left(|U'\cup V'|, \mathcal{F}\right),
        \end{align*}
        a contradiction. 
    \end{proof}
    \begin{claim}\label{CLAIM:r-gp-p-large-tilde-H}
        We have 
        \begin{align}
            \sum_{v\in U}d_{\mathcal{H}}(v) 
            & \le \left(s_1-1+\frac{2\varepsilon}{3}\right)\binom{n}{r-1}, \quad\text{and} \label{equ:r-gp-p-large-U} \\
            \norm{\mathcal{H}_{1} }_p 
            & \ge \left(\left(\frac{\varepsilon}{3}\right)^{1/p} - \left(\frac{\varepsilon}{4}\right)^{1/p}\right)^{p} \binom{n}{r-1}^p. \label{equ:r-gp-p-large-tilde-H}
        \end{align}
    \end{claim}
    \begin{proof}[Proof of Claim~\ref{CLAIM:r-gp-p-large-tilde-H}]
        Let $\mathcal{S}$ be the collection of edges in $\mathcal{H}$ that contain exactly one vertex from $U$. 
        Note that $\mathcal{S} = \mathcal{S}[U, V_1]$ is a semibipartite $r$-graph. 
        Since $F\subseteq K_{s_1, \ldots, s_r}^{r}$ and  $\mathcal{S}$ is $F$-free, it follows from Proposition~\ref{PROP:hypergraph-KST-Zaran} and Claim~\ref{CLAIM:r-gp-p-large-U-upper} that 
        \begin{align*}
            |\mathcal{S}|
            & \le \frac{(s_2 + \cdots + s_{r} - r+1)^{\frac{1}{s_1}}}{r-1} |U| n^{r-1-\frac{1}{s_1 \cdots s_{r-1}}} + (s_1-1)\binom{n}{r-1} \\
            & \le \frac{(s_2 + \cdots + s_{r} - r+1)^{\frac{1}{s_1}}}{r-1} n^{r-1-\frac{1}{s_1 \cdots s_{r-1}} + \delta_2 } + (s_1-1)\binom{n}{r-1} \\
            & \le \frac{\varepsilon}{2} \binom{n}{r-1} + (s_1-1)\binom{n}{r-1}. 
        \end{align*} 
        Let $\mathcal{S}_2$ denote the set of edges in $\mathcal{H}$ that contain at least two vertices from $U$. 
        It is clear that 
        \begin{align*}
            |\mathcal{S}_2|
            \le |U|^2 \binom{n}{r-2}
            \le n^{2\delta_2 } \binom{n}{r-2}
            \le \frac{\varepsilon}{6r}\binom{n}{r-1}.
        \end{align*}
        Therefore,  
        \begin{align}\label{equ:r-gp-large-H-U}
            \sum_{v\in U}d_{\mathcal{H}}(v)
            \le |\mathcal{S}| + r\cdot |\mathcal{S}_2| 
            & \le \frac{\varepsilon}{2} \binom{n}{r-1} + (s_1-1)\binom{n}{r-1} + r\cdot  \frac{\varepsilon}{6r}\binom{n}{r-1} \notag \\
            & = \left(s_1-1+\frac{2\varepsilon}{3}\right) \binom{n}{r-1}.
        \end{align}
        This proves~\eqref{equ:r-gp-p-large-U}. 

        Next, we prove~\eqref{equ:r-gp-p-large-tilde-H}.
        First, note that for every $v\in V_{1}$, we have 
        \begin{align*}
            d_{\mathcal{H}}(v) - d_{\mathcal{H}_{1} }(v) \le |U| \binom{n}{r-2} \le n^{r-2+\delta_2}.
        \end{align*}
        Therefore, by the assumption that $\delta_2 < \frac{p-1}{p}$, we have 
        \begin{align*}
            \sum_{v\in V_{1}} \left(d_{\mathcal{H}}(v) - d_{\mathcal{H}_{1} }(v)\right)^{p}
            \le |V_{1}| \cdot n^{p(r-2+\delta_2 )}
            \le n^{p(r-2+\delta_2 ) + 1}
            \le \frac{\varepsilon}{4} \binom{n}{r-1}^{p}.
        \end{align*}
        Consequently, it follows from Fact~\ref{FACT:Minkowski} that 
        \begin{align*}
            \left(\sum_{v\in V_{1}}d_{\mathcal{H}}^{p}(v)\right)^{1/p}
            & = \left(\sum_{v\in V_{1}} \left(d_{\mathcal{H}_{1} }(v) + d_{\mathcal{H}}(v) - d_{\mathcal{H}_{1} }(v)\right)^p \right)^{1/p} \\
            & \le \left( \sum_{v\in V_{1}}d_{\mathcal{H}_{1} }^{p}(v) \right)^{1/p} 
            + \left( \sum_{v\in V_{1}} \left(  d_{\mathcal{H}}(v) - d_{\mathcal{H}_{1} }(v)\right)^{p}  \right)^{1/p} \\
            & \le \norm{\mathcal{H}_{1} }_{p}^{1/p} 
                + \left(\frac{\varepsilon}{4}\right)^{1/p} \binom{n}{r-1}.
        \end{align*}
        Suppose to the contrary that $\norm{\mathcal{H}_{1} }_{p} < \left(\left(\frac{\varepsilon}{3}\right)^{1/p} - \left(\frac{\varepsilon}{4}\right)^{1/p}\right)^{p} \binom{n}{r-1}^p$. Then it follows from~\eqref{equ:r-gp-large-H-U} and the inequality above that 
        \begin{align*}
            \norm{\mathcal{H}}_{p}
            & = \sum_{v\in U}d_{\mathcal{H}}^{p}(v) 
                + \sum_{v\in V_{1}}d_{\mathcal{H}}^{p}(v) \\
            & < \sum_{v\in U}d_{\mathcal{H}}(v) \cdot \binom{n}{r-1}^{p-1} 
                + \left(\left(\left(\frac{\varepsilon}{3}\right)^{1/p} - \left(\frac{\varepsilon}{4}\right)^{1/p}\right) \binom{n}{r-1}
                + \left(\frac{\varepsilon}{4}\right)^{1/p} \binom{n}{r-1}\right)^{p} \\
            & \le \left(s_1-1+\frac{2\varepsilon}{3}\right) \binom{n}{r-1}^{p} 
                + \frac{\varepsilon}{3} \binom{n}{r-1}^p
            = \left(s_1-1+\varepsilon\right) \binom{n}{r-1}^{p}, 
        \end{align*}
        a contradiction.  
        This proves~\eqref{equ:r-gp-p-large-tilde-H}.
    \end{proof}

    Let $\hat{p} \coloneqq \frac{1-(p-p_{\ast})\delta_1 }{r-1-\alpha} < p_{\ast} < p$. 
    Since $\alpha > r-2$ and $\delta_1 \le \frac{2+\alpha-r}{p-p_{\ast}}$, we have $\hat{p} \ge 1$. 
    It follows from Fact~\ref{FACT:p-q-norm-upper} and \eqref{equ:r-gp-p-large-tilde-H} that there exists a constant $\varepsilon_1>0$ satisfying
    \begin{align*}
        \norm{\mathcal{H}_{1} }_{\hat{p}}
        \ge \frac{\norm{\mathcal{H}_{1} }_{p}}{\left(\Delta(\mathcal{H}_{1} )\right)^{p-\hat{p}}}
        \ge \frac{\norm{\mathcal{H}_{1} }_{p}}{\left(n^{r-1-\delta_1 }\right)^{p-\hat{p}}}
        \ge 
        \frac{\varepsilon_1 n^{p(r-1)}}{\left(n^{r-1-\delta_1 }\right)^{p-\hat{p}}}
        & = \varepsilon_1 n^{\hat{p}(r-1-\alpha) + (p-\hat{p})\delta_1  + \hat{p}\alpha} \\
        & = \varepsilon_1 n^{1-(p-p_{\ast})\delta_1  + (p-\hat{p})\delta_1  + \hat{p}\alpha} \\
        & = \varepsilon_1 n^{1+ \hat{p}\alpha + (p_{\ast}-\hat{p})\delta_1 }. 
    \end{align*}
    Since $(p_{\ast}-\hat{p})\delta_1  > 0$ and $n$ is sufficiently large, we have $\norm{\mathcal{H}_{1} }_{\hat{p}} \gg n^{1+\hat{p}\alpha}$, which, by Proposition~\ref{PROP:r-graph-small-p}, implies that  $\norm{\mathcal{H}_{1}}_{\hat{p}} > \mathrm{ex}_{\hat{p}}(n,\mathcal{F})$, a contradiction. 
    This completes the proof of Proposition~\ref{PROP:r-graph-large-p}.
\end{proof}

\section{Proof of Theorem~\ref{THM:r-graph-critical-point}}\label{SEC:proof-r-gh-critical}
We present the proof of Theorem~\ref{THM:r-graph-critical-point} in this section. 
The following result will be useful for the proof. 
\begin{proposition}\label{PROP:divide-log-bound}
    Let $r\ge 2$ be an integer and $p \ge 1$ be a real number. 
    Suppose that $\mathcal{G} = \mathcal{G}[V_1, \ldots, V_{r}]$ is an $r$-partite $r$-graph with $\min\{|V_i| \colon i \in [r]\} \ge 2$. 
    Then there exists a nonempty set $U \subseteq V_1$ such that 
    \begin{align*}
        |\mathcal{G}[U, V_2, \ldots, V_r]|
        \ge \frac{|U|^{1-\frac{1}{p}}}{2}  \left(\frac{\sum_{v\in V_1}d_{\mathcal{G}}^{p}(v)}{\lceil \log \left(|V_2| \cdots |V_{r}|\right) \rceil} \right)^{1/p} 
        \ge \frac{|U|^{1-\frac{1}{p}}}{4}  \left(\frac{\sum_{v\in V_1}d_{\mathcal{G}}^{p}(v)}{ \log \left(|V_2| \cdots |V_{r}|\right) } \right)^{1/p}.
    \end{align*}
\end{proposition}
\begin{proof}[Proof of Proposition~\ref{PROP:divide-log-bound}]
    Let $N \coloneqq |V_2| \cdots |V_{r}|$ and $t \coloneqq \lceil \log N \rceil$. 
    For each $i \in [t]$, let 
    \begin{align*}
        U_i 
        \coloneqq \left\{v\in V_1 \colon d_{\mathcal{G}}(v) \in [2^{i-1}, 2^{i})\right\}. 
    \end{align*}
    Since $\sum_{v\in V_1}d_{\mathcal{G}}^{p}(v) = \sum_{i\in [t]}\sum_{v\in U_i}d_{\mathcal{G}}^{p}(v)$, 
    by the Pigeonhole Principle, there exists $i_{\ast} \in [t]$ such that 
    \begin{align*}
        \sum_{v\in U_{i_{\ast}}}d_{\mathcal{G}}^{p}(v)
        \ge \frac{\sum_{v\in V_1}d_{\mathcal{G}}^{p}(v)}{t}.
    \end{align*}
    Let $U \coloneqq U_{i^{\ast}}$ and $m \coloneqq |U|$. 
    It follows from the definition of $U_{i_{\ast}}$ that 
    \begin{align*}
        m \cdot 2^{p i_{\ast}}
        = |U|\cdot 2^{p i_{\ast}}
        \ge \sum_{v\in U_{i_{\ast}}}d_{\mathcal{G}}^{p}(v)
        \ge \frac{\sum_{v\in V_1}d_{\mathcal{G}}^{p}(v)}{t}, 
    \end{align*}
    which implies that 
    \begin{align*}
        2^{i_{\ast}-1}
        \ge \frac{1}{2} \left(\frac{\sum_{v\in V_1}d_{\mathcal{G}}^{p}(v)}{m \cdot t} \right)^{1/p}
    \end{align*}
    Therefore, 
    \begin{align*}
        |\mathcal{G}[U, V_2, \ldots, V_r]|
        = \sum_{v\in U}d_{\mathcal{G}}(v)
        \ge m \cdot 2^{i_{\ast}-1} 
        \ge \frac{1}{2} \left(\frac{\sum_{v\in V_1}d_{\mathcal{G}}^{p}(v)}{t} \right)^{1/p} m^{1-\frac{1}{p}}.
    \end{align*}
    This proves Proposition~\ref{PROP:divide-log-bound}. 
\end{proof}

We are now ready to prove Theorem~\ref{THM:r-graph-critical-point}. 
\begin{proof}[Proof of Theorem~\ref{THM:r-graph-critical-point}]
    Recall that $p_{\ast} \coloneqq \frac{1}{r-1-\alpha}$. Let $n$ be sufficiently large. 
    Suppose that $\mathcal{G}$ is an $\mathcal{F}$-free $r$-graph on $n$ vertices. 
    By Proposition~\ref{PROP:r-partite-subgp-p-norm}, there exists a balanced $r$-partition $V_1 \cup \cdots \cup V_{r} = V(\mathcal{G})$ such that the $r$-partite subgraph $\mathcal{H} \coloneqq \mathcal{G}[V_1, \ldots, V_r]$ satisfies
    \begin{align*}
        \norm{\mathcal{H}}_{p_{\ast}}
        \ge \left(\frac{r!}{r^r} + o(1)\right)^{p_{\ast}}\norm{\mathcal{G}}_{p_{\ast}}
        \ge \frac{1}{2}\left(\frac{r!}{r^r}\right)^{p_{\ast}}\norm{\mathcal{G}}_{p_{\ast}}.
    \end{align*}
    Since $\norm{\mathcal{H}}_{p_{\ast}} = \sum_{i\in [r]}\sum_{v\in V_i}d_{\mathcal{H}}^{p_{\ast}}(v)$, by the Pigeonhole Principle, there exists $V_i$ such that 
    \begin{align*}
        \sum_{v\in V_i}d_{\mathcal{H}}^{p_{\ast}}(v)
        \ge \frac{\norm{\mathcal{H}}_{p_{\ast}}}{r}
        \ge \frac{1}{2r}\left(\frac{r!}{r^r}\right)^{p_{\ast}}\norm{\mathcal{G}}_{p_{\ast}}.
    \end{align*}
    By symmetry, we may assume that $i = 1$. 
    
    Applying Proposition~\ref{PROP:divide-log-bound} to $\mathcal{H}$, we obtain a nonempty set $U \subseteq V_1$ of size $m$ for some $m \le |V_1|$ such that 
    \begin{align*}
        |\mathcal{H}[U, V_2, \ldots, V_r]|
        & \ge \frac{|U|^{1-\frac{1}{p_{\ast}}}}{4}  \left(\frac{\sum_{v\in V_1}d_{\mathcal{H}}^{p_{\ast}}(v)}{\log \left(|V_2| \cdots |V_{r}|\right)} \right)^{1/p_{\ast}}  \\
        & \ge \frac{m^{1-\frac{1}{p_{\ast}}}}{4}  \left( \frac{1}{2r}\left(\frac{r!}{r^r}\right)^{p_{\ast}} \frac{\norm{\mathcal{G}}_{p_{\ast}}}{r \cdot \log n} \right)^{1/p_{\ast}} \\
        & = \frac{m^{1+\alpha - (r-1)}}{4}  \left( \frac{1}{2r}\left(\frac{r!}{r^r}\right)^{p_{\ast}} \frac{\norm{\mathcal{G}}_{p_{\ast}}}{r \cdot \log n} \right)^{1/p_{\ast}}.
    \end{align*}
    Since $\mathrm{ex}(n,\mathcal{F}) = O(n^{1+\alpha})$, it follows from Proposition~\ref{PROP:hypergraph-KST-Zaran-b} that $|\mathcal{H}[U, V_2, \ldots, V_r]| \le  C_{\mathcal{F}} m^{1+\alpha -(r-1)} n^{r-1}$ for some constant $ C_{\mathcal{F}} > 0$. 
    Therefore, 
    \begin{align*}
        \frac{m^{1+\alpha - (r-1)}}{4}  \left( \frac{1}{2r}\left(\frac{r!}{r^r}\right)^{p_{\ast}} \frac{\norm{\mathcal{G}}_{p_{\ast}}}{r \cdot \log n} \right)^{1/p_{\ast}}
        \le C_{\mathcal{F}} m^{1+\alpha -(r-1)} n^{r-1}, 
    \end{align*}
    which implies that 
    \begin{align*}
        \norm{\mathcal{G}}_{p_{\ast}}
        \le  C_{\mathcal{F}}^{p_{\ast}}  \cdot 4^{p_{\ast}} \cdot 2r \cdot \left(\frac{r^r}{r!}\right)^{p_{\ast}} r \log n \cdot n^{p_{\ast}(r-1)}
        = C_{\mathcal{F}}^{p_{\ast}} 2^{2p^{\ast}+1} r \left(\frac{r^r}{r!}\right)^{p_{\ast}} n^{p_{\ast}(r-1)} \log n.
    \end{align*}
    This proves Theorem~\ref{THM:r-graph-critical-point}. 
\end{proof}

\section{Proof of Theorem~\ref{THM:graph-critical-point}}\label{SEC:proof-graph-critical}
In this section, we prove Theorem~\ref{THM:graph-critical-point}. 
For convenience, for every integer $\ell \ge 3$, let $C_{\le 2\ell} \coloneqq \{C_4, C_6, \ldots, C_{2\ell}\}$. 
The following two theorems will be useful for us. 
\begin{theorem}[Lam--Verstra{\"e}te~\cite{LV05girth}]\label{THM:girth-upper}
    Let $\ell \ge 3$ be an integer. 
    For every $n \in \mathbb{N}$, 
    \begin{align*}
        \mathrm{ex}(n,C_{\le 2\ell}) 
        \le \frac{1}{2} n^{1+\frac{1}{\ell}} + 2^{\ell^2} n
        = \left(\frac{1}{2} + o(1)\right)n^{1+\frac{1}{\ell}}.
    \end{align*}
\end{theorem}

\begin{theorem}[Naor--Verstra\"{e}te~\cite{Naor05}]\label{THM:NV05-C2k-Zaran}
    Let $\ell \ge 2$ be an integer. 
    Then 
    \begin{align*}
        \mathrm{ex}(m,n,C_{\le 2\ell})
        \le 
        \begin{cases}
            4\left((nm)^{\frac{1}{2}+ \frac{1}{2\ell}} +n+m\right), & \quad\text{if $\ell$ is odd},  \\
           4\left((nm)^{\frac{1}{2}}m^{\frac{1}{\ell}}+n+m\right), & \quad\text{if $\ell$ is even}.
        \end{cases}
    \end{align*}
    In particular, for every $\ell \ge 2$ and for every $n \ge m \ge 1$, 
            \begin{align*}
                \mathrm{ex}(m,n,C_{\le 2\ell})
                \le 4\left((nm)^{\frac{1}{2}+ \frac{1}{2\ell}} +n+m\right),  
            \end{align*}
    and if $m \le n^{\frac{\ell-1}{\ell+1}}$, then $\mathrm{ex}(m,n,C_{\le 2\ell}) \le 4(n+n+m) \le 12 n$.
\end{theorem}

Recall that an ordered sequence of vertices $v_1, \ldots, v_{\ell+1} \in V(G)$ is a \textbf{walk} of length $\ell$ in a graph $G$ if $v_iv_{i+1} \in G$ for all $i \in [\ell]$. 
We use $W_{\ell+1}(G)$ to denote the number of walks of length $\ell$ in $G$. 

The following result will be useful for the proof of Theorem~\ref{THM:graph-critical-point}~\ref{THM:graph-critical-point-1}. 
The case where $k$ is even appears in~{\cite[Theorem~4]{Erdos82}}, while the case where both $k$ and $\ell$ are odd follows from the more general result of Sa{\u g}lam~{\cite[Theorem~1.3]{Saglam2018}}. 

\begin{theorem}[Erd\H{o}s--Simonovits~\cite{Erdos82}, Sa{\u g}lam~\cite{Saglam2018}]\label{THM:path-hom}
    Suppose that $k \ge \ell \ge 1$ are integers such that $k$ is even or $\ell$ is odd. Then for every graph $G$ on $n$ vertices, we have
    \begin{align*}
        \left(\frac{W_{k+1}(G)}{n}\right)^{1/k} 
        \ge  \left(\frac{W_{\ell+1}(G)}{n}\right)^{1/\ell}.
    \end{align*}
\end{theorem}
\begin{proposition}\label{PROP:W4-3/2-norm}
    For every graph $G$ we have 
    \begin{align*}
        W_{4}(G)
        \ge \frac{\norm{G}_{3/2}^{2}}{|V(G)|}. 
    \end{align*}
\end{proposition}
\begin{proof}[Proof of Proposition~\ref{PROP:W4-3/2-norm}]
    It follows from the Cauchy--Schwarz Inequality that
    \begin{align*}
        \left(\sum_{uv\in G} d_{G}^{1/2}(v)\right)^{2}
        & = \left(\sum_{uv\in G}  \left(d_{G}(u) d_{G}(v)\right)^{1/2} \cdot  \left(\frac{1}{d_{G}(u)}\right)^{1/2}\right)^{2} \\
        & \le \left( \sum_{uv\in G} d_{G}(u) d_{G}(v) \right) \cdot \left(\sum_{uv\in G} d_{G}^{-1}(u)\right).
    \end{align*}
    Consequently, 
    \begin{align*}
        W_{4}(G)
        = \sum_{uv\in G} d_{G}(u)d_{G}(v) 
        & \ge \frac{\left(\sum_{uv\in G} d_{G}^{1/2}(v)\right)^{2}}{\sum_{uv\in G} d_{G}^{-1}(u)} \\
        & = \frac{\left(\sum_{v\in V(G)} d_{G}^{1/2}(v) \cdot d_{G}(v)\right)^{2}}{\sum_{u\in V(G)} d_{G}^{-1}(u) \cdot d_{G}(u)} 
        = \frac{\norm{G}_{3/2}^{2}}{|V(G)|}, 
    \end{align*}
    as desired. 
\end{proof}

First, we prove the upper bound for $\mathrm{ex}_{\ell/(\ell-1)}(n,\{C_{4}, \ldots, C_{2\ell}\})$.
\begin{proof}[Proof of Theorem~\ref{THM:graph-critical-point}~\ref{THM:graph-critical-point-1}]
Fix an integer $\ell \ge 3$. 
Let $p \coloneqq \frac{\ell}{\ell-1}$.
Let $C \coloneqq 52 \cdot 2^p < 765/3^{p}$ and let $\varepsilon>0$ be sufficiently small. 
Notice from Proposition~\ref{PROP:r-partite-subgp-p-norm} that for large $n$, 
\begin{align*}
    \mathrm{ex}_{p}(n,C_{\le 2\ell})
     \le \mathrm{ex}_{p}(2n,C_{\le 2\ell})  
     \le (2+o(1))^{p} \cdot \mathrm{ex}_{p}(n,n,C_{\le 2\ell}) 
     \le 3^p \cdot \mathrm{ex}_{p}(n,n,C_{\le 2\ell}). 
\end{align*}
So it suffices to prove that $\mathrm{ex}_{p}(n,n,C_{\le 2\ell}) < C n^{p}$ for all large $n$. 
Suppose to the contrary that this fails. 
Then there exists a $C_{\le 2\ell}$-free bipartite graph $G = G[V_1, V_2]$ with $|V_1| = |V_2| = n$ such that $\norm{G}_p = Cn^{p}$. 
By symmetry, we may assume that 
    \begin{align}\label{equ:even-cycle-V1}
        \sum_{v\in V_1}d_{G}^{p}(v)
        \ge \frac{1}{2} \left(\sum_{v\in V_1}d_{G}^{p}(v) + \sum_{v\in V_2}d_{G}^{p}(v)\right)
        = \frac{\norm{G}_{p}}{2}
        \ge \frac{C}{2}n^{p}. 
    \end{align}
    Let 
    \begin{align*}
        U_1 
        \coloneqq \left\{v\in V_1 \colon d_{G}(v) \ge n^{1-\varepsilon}\right\}
        \quad\text{and}\quad 
        U_2
        \coloneqq \left\{v\in V_1 \colon d_{G}(v) \in [n^{1/\ell + \varepsilon}, n^{1-\varepsilon})\right\}.
    \end{align*}
\begin{claim}\label{CLAIM:girth-U1}
        We have $\sum_{v\in U_1}d_{G}^{p}(v) \le 12 n^p$.
\end{claim}
\begin{proof}[Proof of Claim~\ref{CLAIM:girth-U1}]
        Since $G$ is $C_{\le 2\ell}$-free, it follows from Theorem~\ref{THM:girth-upper} (see also~\cite{AHL02}) that $|G| \le \mathrm{ex}(2n,C_{\le 2\ell}) \le (2^{1/\ell}+o(1))n^{1+1/\ell} \le 2n^{1+1/\ell}$.
        Therefore, 
        \begin{align*}
            |U_1| 
            \le \frac{|G|}{n^{1-\varepsilon}}
            \le \frac{2n^{1+1/\ell}}{n^{1-\varepsilon}}
            = 2n^{1/\ell + \varepsilon}. 
        \end{align*}
        Since $\frac{1}{\ell} + \varepsilon < \frac{\ell-1}{\ell+1}$ for $\ell\ge 3$, it follows from Theorem~\ref{THM:NV05-C2k-Zaran} that 
        \begin{align*}
            |G[U_1, V_2]|
            \le 12 n.
        \end{align*}
        Combining this with Fact~\ref{FACT:p-q-norm-upper}, we obtain 
        \begin{align*}
            \sum_{v\in U_1}d_{G}^{p}(v)
            \le \sum_{v\in U_1}d_{G}(v) \cdot n^{p-1}
            = |G[U_1, V_2]| \cdot n^{p-1}
            \le 12 n^{p},  
        \end{align*}
        which proves Claim~\ref{CLAIM:girth-U1}. 
\end{proof}
\begin{claim}\label{CLAIM:girth-U2}
        We have $\sum_{v\in U_2}d_{G}^{p}(v) \le n^p$.
\end{claim}
\begin{proof}[Proof of Claim~\ref{CLAIM:girth-U2}]
    Let $t \coloneqq \lceil \log n \rceil$. 
    For every $i \in [t]$, let 
    \begin{align*}
        W_i
        \coloneqq \left\{v\in U_2 \colon d_{G}(v) \in [2^{i-1} \cdot n^{1/\ell+\varepsilon}, 2^{i} \cdot n^{1/\ell+\varepsilon})\right\}.
    \end{align*}
    Suppose to the contrary that $\sum_{v\in U_2}d_{G}^{p}(v) > n^p$. 
    Then, it follows from the Pigeonhole Principle that there exists $i \in [t]$ with 
    \begin{align*}
        \sum_{v\in W_i} d_{G}^{p}(v) 
        \ge \frac{\sum_{v\in U_2} d_{G}^{p}(v)}{t}
        \ge \frac{n^p}{t}.
    \end{align*}
    Let $\beta \in [1/\ell+\varepsilon, 1-\varepsilon]$ be the real number such that $n^{\beta} = 2^{i-1} n^{1/\ell+\varepsilon}$. 
    It follows from the definition of $W_i$ that 
    \begin{align*}
        \sum_{v\in W_i} d_{G}(v) 
        \ge \sum_{v\in W_i} \frac{d_{G}^{p}(v)}{(2n^{\beta})^{p-1}}
        = \frac{\sum_{v\in W_i} d_{G}^{p}(v)}{2^{p-1} n^{(p-1)\beta}}
        \ge \frac{n^{p-(p-1)\beta}}{2^{p-1} t}. 
    \end{align*}
    Consequently, 
    \begin{align*}
        |G[W_i, V_2]|
        = \sum_{v\in W_i} d_{G}(v) 
        & = \left(\sum_{v\in W_i} d_{G}(v) \right)^{\frac{1}{2}+\frac{1}{2\ell}}  \left(\sum_{v\in W_i} d_{G}(v) \right)^{\frac{1}{2}-\frac{1}{2\ell}} \\
        & \ge \left(|W_i| \cdot n^{\beta} \right)^{\frac{1}{2}+\frac{1}{2\ell}}  \left(\frac{n^{p-(p-1)\beta}}{2^{p-1} t} \right)^{\frac{1}{2}-\frac{1}{2\ell}}.
    \end{align*}
    Since $p = \frac{\ell}{\ell-1}$ and $\beta \ge \frac{1}{\ell} + \varepsilon$, we have 
    \begin{align*}
        \beta \cdot \left(\frac{1}{2}+\frac{1}{2\ell}\right) + \left(p-(p-1)\beta\right)\cdot  \left(\frac{1}{2}-\frac{1}{2\ell}\right)
        = \frac{1+\beta}{2}
        \ge \frac{1}{2}+\frac{1}{2\ell} + \frac{\varepsilon}{2}. 
    \end{align*}
    Therefore, 
    \begin{align*}
         |G[W_i, V_2]|
         \ge \left(|W_i| \cdot n \right)^{\frac{1}{2}+\frac{1}{2\ell}} \frac{n^{\varepsilon/2}}{\left(2^{p-1} t\right)^{\frac{1}{2}-\frac{1}{2\ell}}}.
    \end{align*}
    Since $\varepsilon > 0$ and $n$ is sufficiently large, it follows from Theorem~\ref{THM:NV05-C2k-Zaran} that $|G[W_i, V_2]| > \mathrm{ex}(|W_i|, |V_2|, C_{\le 2\ell})$, a contradiction.
\end{proof}
Let $V_1' \coloneqq V_{1}\setminus (U_1 \cup U_2)$. 
It follows from~\eqref{equ:even-cycle-V1}, Claims~\ref{CLAIM:girth-U1} and~\ref{CLAIM:girth-U2} that 
    \begin{align}\label{equ:girth-V1'-G}
        \sum_{v\in V_1'}d_{G}^{p}(v)
        & = \sum_{v\in V_1}d_{G}^{p}(v) - \left(\sum_{v\in U_1}d_{G}^{p}(v) + \sum_{v\in U_2}d_{G}^{p}(v)\right)  
         \ge \frac{C}{2}n^{p} - 12 n^{p} - n^{p}
        \ge \frac{C}{4}n^{p}.
    \end{align}
    Let 
    \begin{align*}
        G_1 \coloneqq G[V_1', V_2],
        \quad 
        \tilde{U}
        \coloneqq \left\{v\in V_2 \colon d_{G_1}(v) \ge n^{1/\ell+\varepsilon}\right\},
        \quad\text{and}\quad
        G_2 \coloneqq G[V_1', \tilde{U}].
    \end{align*}
Similar to Claims~\ref{CLAIM:girth-U1} and~\ref{CLAIM:girth-U2}, we have 
    \begin{align*}
        \sum_{v\in \tilde{U}} d_{G_1}^{p}(v)
        \le 12 n^{p} + n^{p} = 13 n^{p}. 
    \end{align*}
    Combining this with Fact~\ref{FACT:p-q-norm-upper}, we obtain 
    \begin{align}\label{equ:girth-V1'}
        \sum_{v\in V_1'}d_{G_2}^{p}(v)
        \le \sum_{v\in V_1'}d_{G_2}(v) \cdot \left(n^{1/\ell+\varepsilon}\right)^{p-1} 
        & = \sum_{u\in \tilde{U}}d_{G_2}(u) \cdot \left(n^{1/\ell+\varepsilon}\right)^{p-1}  \notag \\
        & \le \sum_{u\in \tilde{U}}d_{G_2}^{p}(u)
        \le 13 n^{p}. 
    \end{align}
Let  
    \begin{align*}
        V_2' \coloneqq V_2 \setminus \tilde{U}
        \quad\text{and}\quad 
        H \coloneqq G[V_1', V_2'].
    \end{align*}
It is clear from the definitions of $V_1'$ and $V_2'$ that $\Delta(H) \le n^{1/\ell+\varepsilon}$. 
\begin{claim}\label{CLAIM:graph-critical-cycle-max-deg}
    We have $\norm{H}_p \ge  \sum_{v\in V_1'}d_{H}^{p}(v) \ge 4n^{p}$.
\end{claim}
\begin{proof}[Proof of Claim~\ref{CLAIM:graph-critical-cycle-max-deg}]
    Suppose to the contrary that $\sum_{v\in V_1'}d_{H}^{p}(v) < 4n^{p}$. 
    Then it follows from Fact~\ref{FACT:Minkowski} and~\eqref{equ:girth-V1'} that  
    \begin{align*}
        \sum_{v\in V_1'}d_{G}^{p}(v)
        = \sum_{v\in V_1'}\left(d_{G_2}(v) + d_{H}(v)\right)^{p}
        & \le \left(\left(\sum_{v\in V_1'}d_{G_2}^{p}(v)\right)^{1/p}+ \left(\sum_{v\in V_1'}d_{H}^{p}(v)\right)^{1/p} \right)^{p} \\
        & \le \left(\left(13 n^{p}\right)^{1/p}+ \left(4n^{p}\right)^{1/p} \right)^{p}
        < \frac{C}{4}n^{p}, 
    \end{align*}
    contradicting~\eqref{equ:girth-V1'-G}. Therefore, $\sum_{v\in V_1'}d_{H}^{p}(v) \ge 4n^{p}$. 
\end{proof}

\begin{claim}\label{CLAIM:even-cycle-hom-H}
    We have $W_{\ell+1}(H) \ge 4^{\ell -1}n^2$.
\end{claim}
\begin{proof}[Proof of Claim~\ref{CLAIM:even-cycle-hom-H}]
    It follows from Theorem~\ref{THM:path-hom}, Proposition~\ref{PROP:W4-3/2-norm}, and Corollary~\ref{CORO:p-q-norm-Holder} that 
    \begin{align*}
        \left(\frac{W_{\ell+1}(H)}{v(H)}\right)^{1/\ell}
        \ge \left(\frac{W_{4}(H)}{v(H)}\right)^{1/3}
        & \ge \left(\frac{\norm{H}_{3/2}^{2}}{(v(H))^2}\right)^{1/3} \\
        & = \left(\frac{\norm{H}_{3/2}}{v(H)}\right)^{2/3}
        \ge \left(\frac{\norm{H}_{\frac{\ell}{\ell-1}}}{v(H)}\right)^{\frac{\ell-1}{\ell}}. 
    \end{align*}
    Combining this with Claim~\ref{CLAIM:graph-critical-cycle-max-deg}, we obtain 
    \begin{align*}
        W_{\ell+1}(H) 
        \ge v(H) \cdot \left(\frac{\norm{H}_{\frac{\ell}{\ell-1}}}{v(H)}\right)^{\ell-1}
        = \frac{\norm{H}_{\frac{\ell}{\ell-1}}^{\ell-1}}{\left(v(H)\right)^{\ell-2}}
        \ge \frac{\left(4n^p\right)^{\ell-1}}{n^{\ell-2}}
        = 4^{\ell-1} n^{2}.
    \end{align*}
    This proves Claim~\ref{CLAIM:even-cycle-hom-H}. 
\end{proof}
It follows from Claim~\ref{CLAIM:even-cycle-hom-H} that the number of paths of length $\ell$ in $H$, denoted by $P_{\ell+1}(H)$, satisfies 
\begin{align*}
    P_{\ell+1}(H) 
    & \ge \frac{1}{2}\left( W_{\ell+1}(H) - \binom{\ell+1}{2} \cdot 2n \cdot \left(\Delta(H)\right)^{\ell-1} \right)  \\
    & \ge \frac{1}{2}\left( 4^{\ell-1}n^2 - 2\binom{\ell+1}{2} n^{1+\left(\frac{1}{\ell}+\varepsilon\right)(\ell-1)}  \right)
    > \binom{2n}{2}.
\end{align*}
Therefore, there exist two paths of length $\ell$ that share the same endpoints. 
Since $H$ is bipartite, this implies that $G$ contains a copy of $C_{2i}$ for some $i \in [2, \ell]$, a contradiction. 
\end{proof}

Theorem~\ref{THM:graph-critical-point}~\ref{THM:graph-critical-point-2} is an immediate consequence of Theorem~\ref{THM:graph-critical-point}~\ref{THM:graph-critical-point-1} and the following theorem. 
\begin{theorem}[F{\" u}redi--Naor--Verstra{\" e}te~{\cite[Theorem~3.2]{Furedi06}}]
    Every $C_6$-free bipartite graph $G$ contains a $\{C_4, C_6\}$-free subgraph $H$ such that for every $v \in V(G)$, 
    \begin{align*}
        d_{H}(v)
        \ge \frac{d_{G}(v)}{2}.
    \end{align*}
\end{theorem}

Next, we present the proof of Theorem~\ref{THM:graph-critical-point}~\ref{THM:graph-critical-point-3}. 
The proof is a minor adaption of the Dependent Random Choice (see e.g.~\cite{FS11}).
\begin{proof}[Proof of Theorem~\ref{THM:graph-critical-point}~\ref{THM:graph-critical-point-3}]
Let $F = F[W_1, W_2]$ be a bipartite graph such that $d_{F}(v) \le s$ for every $v\in W_2$. 
Let $t \coloneqq |W_2|$. 
Let $C \coloneqq 2\left(\frac{|V(F)|^{s}}{s!} + |V(F)|\right)$. 
Let $G = G[V_1, V_2]$ be an $n$ by $n$ bipartite graph with  $\norm{G}_s \ge C n^s$.

By symmetry, we may assume that $\sum_{v\in V_1}d_{G}^{s}(v) \ge \frac{\norm{G}_s}{2} \ge \frac{C n^s}{2}$. 
Choose uniformly at random $s$ vertices (with repetations allowed) $v_1, \ldots, v_{s}$ from $V_2$. 
Let $\mathbf{X} \coloneqq N_{G}(v_1) \cap \cdots \cap N_{G}(v_{s}) \subseteq V_1$. 
It is easy to see that 
\begin{align*}
    \mathbb{E}\left[|\mathbf{X}|\right] 
    = \sum_{v\in V_1} \left(\frac{d_{G}(v)}{n}\right)^{s}
    = \frac{\sum_{v\in V_1}d_{G}^{s}(v)}{n^{s}} 
    \ge \frac{C n^{s}/2}{n^{s}}
    = \frac{C}{2}. 
\end{align*}
We call an $s$-set in $\mathbf{X}$ \textbf{bad} if it has  at most $t$ common neighbors. 
Let $\mathbf{Y}$ denote the collection of bad $s$-sets in $\mathbf{X}$. 
Notice that an $s$-set $S$ is contained in $\mathbf{X}$ only if $\{v_1, \ldots, v_{s}\} \subseteq \bigcap_{u\in S}N_{G}(u)$. 
Therefore, 
\begin{align*}
    \mathbb{E}\left[|\mathbf{Y}|\right]
    \le \binom{|V_1|}{s}\left(\frac{t}{n}\right)^{s} 
    \le \frac{t^s}{s!}. 
\end{align*}
It follows that 
\begin{align*}
    \mathbb{E}\left[|\mathbf{X}|-|\mathbf{Y}|\right] 
    \ge \frac{C}{2} - \frac{t^s}{s!} \ge t.
\end{align*}
By deleting one vertex from each bad set, we see that there exists a selection of $s$ vertices $\{v_1, \ldots, v_{s}\} \subseteq V_2$ along with a set $X' \subseteq N_{G}(v_1) \cap \cdots \cap N_{G}(v_{s}) \subseteq V_1$ of size at least $t$ such that every $s$-subset of $X'$ has at least $t$ common neighbors.
It is clear that $F$ can be greedily embedded into $G[X', V_2]$ with $W_1 \subseteq X'$ and $W_2 \subseteq V_2$, a contradiction. 
\end{proof}

\section{Concluding remarks}\label{SEC:remarks}
Let $F$ be an $r$-partite $r$-graph satisfying $\mathrm{ex}(n,F) = O(n^{1+\alpha})$. 
Recall from Fact~\ref{FACT:r-gp-p-norm-lower-bound} and Theorem~\ref{THM:main-r-graph} that for every $p > \frac{1}{r-1-\alpha}$, we have 
\begin{align*}
    \left(\tau_{\mathrm{ind}}(F) - 1 + o(1)\right) \binom{n}{r-1}^{p}
    \le \mathrm{ex}_{p}(n,F) 
    \le \left(\tau_{\mathrm{part}}(F) - 1 + o(1)\right) \binom{n}{r-1}^{p}.
\end{align*}
Additionally, recall that this provides an asymptotically tight bound for $\mathrm{ex}_{p}(n,F)$ in the case $r=2$, as $\tau_{\mathrm{ind}}(F) = \tau_{\mathrm{part}}(F)$ for every bipartite graph $F$. Unfortunately, the equality $\tau_{\mathrm{ind}}(F) = \tau_{\mathrm{part}}(F)$ does not necessarily hold for $r \ge 3$, as shown by the following example. 

Let $F$ denote the $3$-graph with vertex set $\{a,b,c,a_1,a_2,a_3,b_1,b_2,b_3,c_1,c_2,c_3 \}$ 
and edge set 
\begin{align*}
    \left\{\{a, b_i, c_j\}\colon  (i,j)\in[3]^2 \right\} 
    \cup \left\{\{a_i, b, c_j\} \colon  (i,j)\in[3]^2 \right\} 
    \cup \left\{\{a_i, b_j, c\} \colon (i,j)\in[3]^2 \right\}.
\end{align*}
It is easy to verify that $\tau_{\mathrm{part}}(F) = 4$ while $\tau_{\mathrm{ind}}(F) = 3$ (with $\{a,b,c\}$ serving as a witness). 

\begin{problem}\label{PROB:r-gp-p-parge}
    Let $r \ge 3$. Suppose that $\mathcal{F}$ is a degenerate family of $r$-graphs satisfying $\mathrm{ex}(n,\mathcal{F}) = O(n^{1+\alpha})$ for some constant $\alpha > 0$. Determine whether $\lim\limits_{n\to \infty} \mathrm{ex}_{p}(n,\mathcal{F})/n^{p(r-1)}$ exists for $p > \frac{1}{r-1-\alpha}$, and, if so, find its value.
\end{problem}

On the other hand, drawing parallels to the Exponent Conjecture of Erd\H{o}s--Siminovits, we propose the following bold conjecture for hypergraphs, which, if true, would show that Theorem~\ref{THM:main-r-graph} is tight in the exponent for the case $p < \frac{1}{r-1-\alpha}$ as well. 
\begin{conjecture}\label{CONJ:exponent-r-gp}
    Let $r \ge 3$. Suppose that $\mathcal{F}$ is a degenerate finite family of $r$-graphs satisfying $\mathrm{ex}(n,\mathcal{F}) = \Omega(n^{1+\alpha})$ for some constant $\alpha > r-2$.
    Then there exist constants $\beta > 0$, $c>0$, and $C>0$ such that for all sufficiently large $n$, 
    \begin{align*}
        c
        \le \frac{\mathrm{ex}(n,\mathcal{F})}{n^{1+\beta}}
        \le C. 
    \end{align*}
\end{conjecture}
\textbf{Remark.}
Several results such as those in~\cite{MYZ18,PZ21} provide some evidence supporting this conjecture. On the other hand, examples in~\cite{RS78,FG21} show that the requirement  $\alpha > r-2$ cannot be removed in general.

Recall from Theorem~\ref{THM:r-graph-critical-point} that we provided a general upper bound for $\mathrm{ex}_{p}(n,\mathcal{F})$ when $p$ is the threshold. An interesting problem is to explore whether the $\log n$ factor can be removed from this upper bound.
\begin{problem}\label{PROP:r-gp-critical}
    Let $r \ge 2$ be an integer. 
    Suppose that $\mathcal{F}$ is a degenerate family of $r$-graphs satisfying $\mathrm{ex}(n,\mathcal{F}) = O(n^{1+\alpha})$ for some constant $\alpha > 0$. 
    Is it true that  
    \begin{align*}
        \mathrm{ex}_{p_{\ast}}(n,\mathcal{F})
        = O\left(n^{p_{\ast}(r-1)}\right)
        \quad\text{for}\quad p_{\ast} = \frac{1}{r-1-\alpha}~\text{?}
    \end{align*}    
\end{problem}
Given integers $r > t \ge 1$ and a real number $p > 0$, let the \textbf{$(t,p)$-norm} of an $r$-graph $\mathcal{H}$ be defined as
\begin{align*}
    \norm{\mathcal{H}}_{t,p}
    \coloneqq \sum_{T\in \binom{V(\mathcal{H})}{t}} d_{\mathcal{H}}^{p}(T).
\end{align*}
Similarly, for a family $\mathcal{F}$ of $r$-graphs, define the \textbf{$(t,p)$-norm Tur\'{a}n number} of $\mathcal{F}$ as 
\begin{align*}
    \mathrm{ex}_{t,p}(n,\mathcal{F})
    \coloneqq \max\left\{\norm{\mathcal{H}}_{t,p} \colon \text{$v(\mathcal{H}) = n$ and $\mathcal{H}$ is $\mathcal{F}$-free}\right\}.
\end{align*}
The $(t,p)$-norm Tur\'{a}n number $\mathrm{ex}_{t,p}(n,\mathcal{F})$ was systematically studied in~\cite{CIDLLP24} for nondegenerate families $\mathcal{F}$. 
However, many degenerate cases remain unexplored. 
\begin{problem}\label{PROB:tp-norm-r-partite}
    Let $r > t \ge 2$ be integers and $\mathcal{F}$ be a finite family of $r$-partite $r$-graphs such that $\mathrm{ex}(n,\mathcal{F}) = n^{\beta + o(1)}$.
    Determine the exponent of $\mathrm{ex}_{t,p}(n,\mathcal{F})$ for all $p > 1$.
\end{problem}
Given two graphs $Q$ and $G$, we use $N(Q, G)$ to denote the number of copies of $Q$ in $G$.
The \textbf{generalized Tur\'{a}n number} $\mathrm{ex}(n,Q,\mathcal{F})$ is the maximum number of copies of $Q$ in an $n$-vertex $\mathcal{F}$-free graph. 
The generalized Tur\'{a}n problem was first considered by Erd\H{o}s in~\cite{E62}, and was systematically studied by Alon--Shikhelman in~\cite{AS16}. 

Given integers $p \ge r > t \ge 0$, the \textbf{$(r,t)$-book with $p$-pages}, denoted by $B_{t,r,p}$, is the graph constructed as follows$\colon$
\begin{itemize}
    \item Take $p$ sets $V_{1}, \ldots, V_{p}$, each of size $r$, such that there exists a $t$-set $C$ satisfying $V_i \cap V_j = C$ for all $1 \le i < j \le p$. 
    \item Place a copy of $K_{r}$ on each $V_i$. 
\end{itemize}
Observe that $B_{1,2,p}$ is simply a star graph with $p$ edges. 

In parallel, one could define the \textbf{$(t,r,p)$-norm} of a graph as follows$\colon$ 
Given a graph $G$ and a $t$-set $S \subseteq V(G)$ that induces a copy of $K_t$, let $d_{G,r}(S)$ denote the number of copies of $K_r$ in $G$ that contains $S$. 
Let 
\begin{align*}
    \norm{G}_{t,r,p}
    \coloneqq \sum d_{G,r}^{p}(S),
\end{align*}
where the summation is taken over all $t$-subsets $S \subseteq V(G)$ that induce a copy of $K_t$ in $G$. 
Similarly, let 
\begin{align*}
    \mathrm{ex}_{t,r,p}(n,F)
    \coloneqq \max\left\{\norm{G}_{t,r,p} \colon \text{$v(G) = n$ and $G$ is $F$-free}\right\}.
\end{align*}
One could consider extending results in this paper to the function $\mathrm{ex}_{t,r,p}(n,F)$. This will provide an upper bound for the generalized Tur\'{a}n number $\mathrm{ex}(n,B_{t,r,p},F)$, since for every graph $G$, 
\begin{align*}
    N(B_{t,r,p}, G)
    = \sum \binom{d_{G,r}(S)}{p}
    \le \frac{1}{p!} \sum d_{G,r}^{p}(S)
    = \frac{\norm{G}_{t,r,p}}{p!},
\end{align*}
where the summation is taken over all $t$-subsets $S \subseteq V(G)$ that induce a copy of $K_t$ in $G$.

For a bipartite graph $G[V_1, V_2]$ with parts $V_1$ and $V_2$, define
\begin{align*}
    \norm{G}_{p,\mathrm{left}}
    \coloneqq \sum_{v\in V_1}d_{G}^{p}(v)
    \quad\text{and}\quad 
    \norm{G}_{p,\mathrm{right}}
    \coloneqq \sum_{v\in V_2}d_{G}^{p}(v).
\end{align*}
Note that $\norm{G}_{1,\mathrm{left}} = \norm{G}_{1,\mathrm{right}} = |G|$ and $\norm{G}_{p} = \norm{G}_{p,\mathrm{left}} + \norm{G}_{p,\mathrm{right}}$ for every $p \ge 1$.

An important variation of the Tur\'{a}n problem is the Zarankiewicz problem. 
Given bipartite graphs $F$ and $G$ with fixed bipartitions $V(F) = W_1\cup W_2$ and $V(G) = V_1 \cup V_2$, an \textbf{ordered copy} of $F[W_1, W_2]$ in  $G[V_1, V_2]$ is a copy of $F$ where $W_1$ is contained in $V_1$ and $W_2$ is contained in $V_2$. 
Given integers $m,n \ge 1$, the \textbf{Zarankiewicz number} $Z(m,n,F[W_1, W_2])$ is the maximum number of edges in a bipartite graph $G = G[V_1, V_2]$ with $|V_1| = m$ and $|V_2| = n$ that does not contain an ordered copy of $F[W_1, W_2]$. 

Extending the Zarankiewicz number to the $p$-norm, for every $p \ge 1$, let $Z_{p,\mathrm{left}}(m,n,F[W_1, W_2])$ (resp. $Z_{p,\mathrm{right}}(m,n,F[W_1, W_2])$) denote the maximum value of $\norm{G}_{p,\mathrm{left}}$ (resp. $\norm{G}_{p,\mathrm{right}}$) over all bipartite graphs $G = G[V_1, V_2]$ with $|V_1| = m$ and $|V_2| = n$ that do not contain an ordered copy of $F[W_1, W_2]$.
When the order $[W_1, W_2]$ is clear from the context, for simplicity, we will use $Z(m, n, F)$, $Z_{p,\mathrm{left}}(m, n, F)$, and $Z_{p,\mathrm{right}}(m, n, F)$ to represent $Z(m, n, F[W_1, W_2])$, $Z_{p,\mathrm{left}}(m, n, F[W_1, W_2])$, and $Z_{p,\mathrm{right}}(m, n, F[W_1, W_2])$ respectively.

The following theorem can be derived through relatively straightforward modifications to the proofs presented in this paper, so we omit the details. 
\begin{theorem}\label{THM:Zaran-threshold}
     Suppose that $F = F[W_1,W_2]$ is a bipartite graph such that
     $Z(m,n,F) = O(m^{\alpha} n^{\beta} + n + m)$ for some constants $\alpha, \beta \in (0,1)$ and every $n,m \ge 1$. 
     Then there exists a constant $C_{\mathcal{F}} = C_{\mathcal{F}}(p) > 0$ such that 
     \begin{align*}
         Z_{p,\mathrm{left}}(m,n,F)
         \le  
         \begin{cases}
             C_{\mathcal{F}}\left(m^{1-p(1-\alpha)} n^{\beta p} + (m+n^{p})\log^{\frac{p_{\ast}-1}{\delta}+1} n  \right), & \quad\text{if}\quad p \in \left[1, \frac{1}{2-\alpha-\beta}\right), \\
             C_{\mathcal{F}}\left(m^{1-p(1-\alpha)} n^{p \beta} + m + n^{p}\right) \log n, & \quad\text{if}\quad p = \frac{1}{2-\alpha-\beta}, \\
             \left(\tau_{\mathrm{ind}}(\mathcal{F})-1\right) n^{p} + o_{n}(n^{p}) + o_{m}(m^{p}), & \quad\text{if}\quad p > \frac{1}{2-\alpha-\beta}. 
         \end{cases}
     \end{align*}
     and 
     \begin{align*}
         Z_{p,\mathrm{right}}(m,n,F)
         \le  
         \begin{cases}
             C_{\mathcal{F}}\left(m^{\alpha p} n^{1 - p(1-\beta)} + (m^p + n) \log^{\frac{p_{\ast}-1}{\delta}+1} m \right), & \quad\text{if}\quad p \in \left[1, \frac{1}{2-\alpha-\beta}\right), \\
             C_{\mathcal{F}}\left(m^{p\alpha} n^{1-p(1-\beta)} + m^{p} + n\right) \log m, & \quad\text{if}\quad p = \frac{1}{2-\alpha-\beta}, \\
             \left(\tau_{\mathrm{ind}}(\mathcal{F})-1\right) m^{p} + o_{n}(n^{p}) + o_{m}(m^{p}), & \quad\text{if}\quad p > \frac{1}{2-\alpha-\beta}. 
         \end{cases}
     \end{align*}
     Here, $p_{\ast} \coloneqq \frac{1}{2-\alpha-\beta}$ and $\delta \coloneqq \frac{1-p(2-\alpha - \beta)}{2}$.
\end{theorem}
\textbf{Remark.}
By summing $Z_{p,\mathrm{left}}(m,n,F)$ and $Z_{p,\mathrm{right}}(m,n,F)$, we obtain an upper bound for $\mathrm{ex}_{p}(m,n,F)$ and the two-sided $Z_{p}(m,n,F[W_1, W_2])$. Here, $Z_{p}(m,n,F[W_1, W_2])$ represents the maximum $p$-norm of a bipartite graph $G = G[V_1, V_2]$ with $|V_1| = m$ and $|V_2| = n$ that does not contain an ordered copy of $F[W_1, W_2]$. 

\section*{Acknowledgements}
We would like to thank D\'{a}niel Gerbner for informing us about~\cite{FK06}.
\bibliographystyle{alpha}
\bibliography{LpNorm}
\begin{appendix}
\section{Sketch of the Proof of Theorem~\ref{THM:Zaran-threshold}}
In this section, we provide a sketch of the proof for Theorem~\ref{THM:Zaran-threshold}. 
\begin{proposition}\label{PROP:Zaran-general}
    For every $p \ge 1$, we have 
    \begin{align*}
        \norm{G}_{p, \mathrm{left}}
        = O\left(m^{1 - p(1-\alpha)} n^{p \beta} \log n + (m+n^{p}) \log n \right).
    \end{align*}
\end{proposition}
\begin{proof}[Proof of Proposition~\ref{PROP:Zaran-general}]
    This is a direct consequence of Proposition~\ref{PROP:divide-log-bound}. 
\end{proof}
\begin{proposition}\label{PROP:Zaran-p-large}
    For every $p > \frac{1}{2-\alpha - \beta}$, we have 
    \begin{align*}
        \norm{G}_{p, \mathrm{left}}
        = \left(\tau_{\mathrm{ind}}(\mathcal{F})-1\right) n^{p} + o_{n}(n^{p}) + o_{m}(m^{p}).
    \end{align*}
\end{proposition}
\begin{proof}[Proof of Proposition~\ref{PROP:Zaran-p-large}]
    The proof is similar to that of Proposition~\ref{PROP:r-graph-large-p}. 
\end{proof}
\begin{proposition}\label{PROP:Zaran-p-small}
    For every $1 < p < \frac{1}{2-\alpha - \beta}$, we have 
    \begin{align*}
        \norm{G}_{p, \mathrm{left}}
        = O\left(m^{1-p(1-\alpha)} n^{\beta p} + (m+n^{p}) \log^{\frac{p_{\ast}-1}{\delta} + 1} n  \right),
    \end{align*}
    where $p_{\ast} \coloneqq \frac{1}{2-\alpha - \beta}$ and $\delta \coloneqq \frac{1-p(2-\alpha-\beta)}{2}$. 
\end{proposition}
\begin{proof}[Sketch proof of Proposition~\ref{PROP:Zaran-p-small}]
    Let $G = G[V_1, V_2]$ be a bipartite graph with $|V_1| = m$ and $|V_2| = n$. Let $\varepsilon \coloneqq \delta/2 = 1 - p(2-\alpha - \beta)$. 
    Suppose that $G$ does not contain any ordered copy of $F = F[W_1, W_2]$. 
    Suppose to the contrary that 
    \begin{align*}
        \norm{G}_{p, \mathrm{left}}
        \ge C \left(m^{1-p(1-\alpha)} n^{\beta p} + (m+n^{p})\log^{\frac{p_{\ast}-1}{\delta} + 1} n  \right), 
    \end{align*}
    where $C > 0$ is a large constant. 
%
%

    We may assume that 
    \begin{align}\label{equ:Zaran-assum-m-n}
        m^{1 - p(1-\alpha)} n^{p \beta} \ge (m+n^{p}) \log^{\frac{p_{\ast}-1}{\delta}} n, 
    \end{align}
    since otherwise, by Proposition~\ref{PROP:Zaran-general}, we are done. 
    Similar to the proof of Lemma~\ref{LEMMA:regularization-p-norm}, for every bipartite graph $H = H[U_1, U_2]$, let 
    \begin{align*}
        \Phi(H)
        \coloneqq \frac{\norm{H}_{p,\mathrm{left}}}{|U_1|^{1-p(1-\alpha)} |U_2|^{p\beta}}.  
    \end{align*}

    Define a sequence of bipartite graphs $G_i = G_{i}[A_i, B_i]$ for $1 \le i \le k$ for some integer $k$ with 
    \begin{enumerate}[label=(\roman*)]
        \item\label{item:Zaran-1} $A_0 = V_1 \supseteq A_1 \supseteq \cdots \supseteq A_k$ and $B_0 = V_2 \supseteq B_1 \supseteq \cdots \supseteq B_k$, 
        \item\label{item:Zaran-2} $|A_i| = \frac{|V_1|}{2^{i}}$ and $|B_i| = \frac{|V_2|}{2^{i}}$ for $i \in [k]$,
        \item\label{item:Zaran-3} $\Phi(G_{i+1}) \ge 2^{\delta - \varepsilon} \cdot \Phi(G_{i})$ for $i \in [0,k-1]$, 
        \item\label{item:Zaran-4} for every $i \in [k-1]$, we have $\sum_{u \in U_i} d_{G_{i-1}}^{p}(u) \ge \frac{1}{2^{\varepsilon}} \norm{G_{i-1}}_{p,\mathrm{left}}$, where
        \begin{align*}
            U_{i} 
            \coloneqq \left\{v\in A_{i-1} \colon d_{G_{i-1}}(v) \ge \left(\frac{2\norm{G_{i-1}}_{p,\mathrm{left}}}{|A_{i-1}|}\right)^{1/p}\right\},
        \end{align*} 
        \item\label{item:Zaran-5} we have $\sum_{u \in U_k} d_{G_{k-1}}^{p}(u) < \frac{1}{2^{\varepsilon}} \norm{G_{k-1}}_{p,\mathrm{left}}$, where
        \begin{align*}
            U_{k} 
            \coloneqq \left\{v\in A_{k-1} \colon d_{G_{k-1}}(v) \ge \left(\frac{2\norm{G_{k-1}}_{p,\mathrm{left}}}{|A_{k-1}|}\right)^{1/p}\right\}.
        \end{align*} 
    \end{enumerate}
    \begin{claim}\label{CLAIM:Zaran-k-ast}
        We have $k \le k_{\ast} \coloneqq \frac{2}{\delta} \log\log n$.
    \end{claim}
    \begin{proof}[Proof of Claim~\ref{CLAIM:Zaran-k-ast}]
        Notice from~\eqref{equ:Zaran-assum-m-n} that for every $i \le k_{\ast}$, we have  
        \begin{align}\label{equ:Zaran-assum-m-n-i}
            |A_{i}|^{1-p(1-\alpha)} |B_{i}|^{p \beta}
            & = \left(\frac{m}{2^{i}}\right)^{1-p(1-\alpha)} \left(\frac{n}{2^{i}}\right)^{p \beta}
            = \frac{m^{1-p(1-\alpha)} n^{p \beta}}{2^{i \cdot (1-p(1-\alpha) +p\beta)}}
            \ge \frac{m^{1-p(1-\alpha)} n^{p \beta}}{2^{i \cdot p_{\ast}}} \notag \\
            & \ge \frac{(m+n^{p}) 2^{k_{\ast}(p_{\ast}-1)}}{2^{i \cdot p_{\ast}}}
            \ge \frac{m+n^{p}}{2^{i}}
            \ge |A_{i}| + |B_{i}|^{p}.
        \end{align}
        Suppose to the contrary that $k > k_{\ast}$. 
        Then it follows from~\ref{item:Zaran-3} that 
        \begin{align*}
            \frac{\norm{G_{k_{\ast}}}_{p, \mathrm{left}}}{|A_{k}|^{1-p(1-\alpha)} |B_{k}|^{p \beta}}
            = \Phi(G_{k})
            \ge \left(2^{\delta - \varepsilon}\right)^{k_{\ast}} \cdot \Phi(G)
            = 2^{\frac{\delta}{2} \cdot k_{\ast}} \cdot \Phi(G)
            \ge C \log n,
        \end{align*}
        which, combining with~\eqref{equ:Zaran-assum-m-n-i}, implies that 
        \begin{align*}
            \norm{G_{k_{\ast}}}_{p, \mathrm{left}} 
            \ge C |A_{k}|^{1-p(1-\alpha)} |B_{k}|^{p \beta} \log n
            \ge \frac{C}{2} \left(|A_{k}|^{1-p(1-\alpha)} |B_{k}|^{p \beta} + |A_{k}| + |B_{k}|^{p}\right) \log n, 
        \end{align*}
        contradicting Proposition~\ref{PROP:Zaran-general}. 
    \end{proof}
    Since $\sum_{u \in U_k} d_{G_{k-1}}^{p}(u) < \frac{1}{2^{\varepsilon}} \norm{G_{k-1}}_{p,\mathrm{left}}$ (due to~\ref{item:Zaran-5}), we have 
    \begin{align*} 
        \sum_{u \in A_{k} \setminus U_k} d_{G_{k-1}}^{p}(u)
        = \norm{G_{k-1}}_{p,\mathrm{left}} - \sum_{u \in U_k} d_{G_{k-1}}^{p}(u) 
        \ge \left(1 - \frac{1}{2^{\varepsilon}}\right) \norm{G_{k-1}}_{p,\mathrm{left}}. 
    \end{align*}
    On the other hand, it follows from the definition of $U_k$ and~\eqref{equ:Zaran-assum-m-n} that 
    \begin{align*}
        \sum_{u \in A_{k} \setminus U_k} d_{G_{k-1}}^{p}(u)
        & \le \sum_{u \in A_{k} \setminus U_k} d_{G_{k-1}}(u) \cdot \left(\frac{2 \norm{G_{k-1}}_{p, \mathrm{left}}}{|A_{k-1}|}\right)^{\frac{p-1}{p}}  \\
        & \le |G_{k-1}| \cdot \left(\frac{2 \norm{G_{k-1}}_{p, \mathrm{left}}}{|A_{k-1}|}\right)^{\frac{p-1}{p}}  \\
        & \le C\left(|A_{k-1}|^{\alpha} |B_{k-1}|^{\beta} + |A_{k-1}| + |B_{k-1}|\right) \cdot \left(\frac{2 \norm{G_{k-1}}_{p, \mathrm{left}}}{|A_{k-1}|}\right)^{\frac{p-1}{p}} \\
        & \le 2C |A_{k-1}|^{\alpha} |B_{k-1}|^{\beta} \cdot \left(\frac{2 \norm{G_{k-1}}_{p, \mathrm{left}}}{|A_{k-1}|}\right)^{\frac{p-1}{p}}. 
    \end{align*}
    Therefore, 
    \begin{align*}
        \left(1 - \frac{1}{2^{\varepsilon}}\right) \norm{G_{k-1}}_{p,\mathrm{left}}
        \le 2C |A_{k-1}|^{\alpha} |B_{k-1}|^{\beta} \cdot \left(\frac{2 \norm{G_{k-1}}_{p, \mathrm{left}}}{|A_{k-1}|}\right)^{\frac{p-1}{p}}, 
    \end{align*}
    which implies that 
    \begin{align*}
        \norm{G_{k-1}}_{p,\mathrm{left}} 
        \le \frac{2^{2p-1}C^{p}}{(1-2^{-\varepsilon})^{p}} |A_{k-1}|^{1 - p(1-\alpha)} |B_{k-1}|^{p \beta}. 
    \end{align*}
    Combining this with~\ref{item:Zaran-3}, we obtain 
    \begin{align*}
        \norm{G}_{p,\mathrm{left}}
        & = \Phi(G) \cdot m^{1-p(1-\alpha)} n^{p \beta} 
        \le \Phi(G_{k-1}) \cdot m^{1-p(1-\alpha)} n^{p \beta} \\
        & = \frac{\norm{G_{k-1}}_{p,\mathrm{left}}}{|A_{k-1}|^{1 - p(1-\alpha)} |B_{k-1}|^{p \beta}} \cdot m^{1-p(1-\alpha)} n^{p \beta} 
         \le \frac{2^{2p-1}C^{p}}{(1-2^{-\varepsilon})^{p}} \cdot m^{1-p(1-\alpha)} n^{p \beta}. 
    \end{align*}
    This completes the proof of Proposition~\ref{PROP:Zaran-p-small}. 
\end{proof}

\end{appendix}
\end{document}